\documentclass[11pt,reqno]{amsart}
\usepackage{amsthm}
\usepackage[english]{babel}
\usepackage{amssymb}
\usepackage{epsfig}
\usepackage{graphicx, color}
\usepackage{hyperref}
\usepackage{natbib}
\let\cite=\citet
\usepackage{hyperref}

\newcommand{\NN}{{\mathbb N}}
\newcommand{\ZZ}{{\mathbb Z}}

\newcommand{\EE}{{\mathbb E}}

\newcommand\1{\leavevmode\hbox{\rm \small1\kern-0.35em\normalsize1}}

\usepackage{pgf}
\usepackage{tikz}
\usepackage{tikz-qtree}
\usetikzlibrary{trees,arrows}
\usetikzlibrary{shapes,positioning}

\def\A{{\mathcal A}}

\def\F{{\mathcal F}}

\def\M{{\mathcal M}}

\newcommand{\XX}{{\mathcal X}}

\renewcommand{\ge}{\geqslant}
\makeatletter
\def\blfootnote{\xdef\@thefnmark{}\@footnotetext}
\makeatother

\numberwithin{equation}{section}

\theoremstyle{plain}

\newtheorem{Theo}{Theorem}

\newtheorem{Prop}{Proposition}
\newtheorem{Cor}{Corollary}
\newtheorem{Lem}{Lemma}
\newtheorem{Claim}{Claim}

\newtheorem{Not}{Notation Alert}

\begin{document}

\title[Non-regular $g$-measures and variable length memory chains]{Non-regular $g$-measures and variable length memory chains}
\author{Ricardo F. Ferreira}
\author{Sandro Gallo}
\author{Fr\'ed\'eric Paccaut}
\begin{abstract}
It is well-known that there always exists at least one stationary measure compatible with a continuous $g$-function $g$. Here we prove that if the set of discontinuities of a $g$-function $g$ has null measure under a candidate measure obtained by some asymptotic procedure, then this candidate measure is compatible with $g$. We explore several implications of this result, and discuss comparisons with the literature concerning assumptions and examples. 

Important part of the paper is dedicated to the case of variable length memory chains, for which we obtain existence, uniqueness and weak-Bernoullicity (or $\beta$-mixing) under new assumptions. These results are specially designed for variable length memory models, and do not require uniform continuity.

We also provide a further discussion on some related notions, such as random context processes,  non-essential discontinuities and  everywhere discontinuous stationary measures.

\end{abstract}

\maketitle

\section{Introduction}

The present paper is about $\mathbb Z$-indexed stationary stochastic processes with finite alphabet, or equivalently, the translation invariant measures on space of bi-infinite sequences of symbols. We are mainly interested in the questions of existence, uniqueness and mixing properties of such measures when we are given a set of conditional rules, that is, we consider the following basic questions: given a set of transition probabilities (or a probability kernel) from ``infinite pasts'' to ``present symbols'' of the alphabet,
\begin{enumerate}
\item does it exist a stationary process whose conditional probabilities are given by this kernel?
\item if yes, is it unique, and what more can we say about its statistical properties?
\end{enumerate}
Naturally, a well-known special case is that of Markov chains (of any finite order), for which the set of transition probabilities is given by a transition matrix. For such models, the question of existence is automatically solved ({recall that we consider finite alphabets}), and the question of uniqueness is simply a question of irreducibility. In the present paper, we rather  consider non-Markovian cases, which means that we are given sets of  transition probabilities  that may depend on unbounded parts of the pasts. Chains specified by such sets of transition probabilities would therefore have unbounded, or even possibly infinite, memory.  

Such objects were first studied in the literature of stochastic processes under the name \emph{cha\^ines \`a liaisons compl\`etes} (chains with complete connections) coined in the seminal works of \cite{onicescu/mihoc/1935} and \cite{doeblin/fortet/1937}. Since the 70's and the work of \cite{keane/1972}, it has been paralleled by the literature of $g$-measures in the field of ergodic theory \citep[just to mention some of them]{Ledrappier,walters/1975,hulse/1991,johansson/oberg/2003}. The set of transition probabilities is sometimes called  probability kernel and usually denoted by $P$ in the stochastic processes literature, while in the literature of ergodic theory, it is called a $g$-function, sometimes itself denoted by $g$. In any way, in our terminology,  chains with infinite memory and $g$-measures are the same objects: the $g$-measure is simply the law of the stationary stochastic process having function $g$ as set of transition probabilities.

\vspace{0.3cm}

Question 1 and 2 above were answered, in both literatures, essentially under the assumption that the $g$-function is continuous (this implies existence), strongly positive, and  eventually with a rapidly vanishing variation if we wish to have uniqueness. In the last 10 years, several works \citep{Desantis/piccioni/2012,gallo/garcia/2013,gallo/paccaut/2013,imbuzeiro/2015} have studied the case of possibly discontinuous $g$-functions (kernels) under several perspectives, including the existence problem, the uniqueness problem and further properties such as perfect simulation, mixing properties and statistical inference. 
The present paper is more in the vein  (and indeed can be considered a sequel)  of \cite{gallo/paccaut/2013}.

With respect to question 1 stated above, loosely speaking, the main problem is to find a characterisation of the smallness of the set $\mathcal{D}_g$ of discontinuities of $g$  guaranteeing existence.  For instance, \cite{gallo/paccaut/2013} asked that this set has negative topological pressure (see \eqref{eq:top_entropy} for a definition). Having in mind the fact that  most examples of stationary measures in the literature give measure 0 to $\mathcal{D}_g$, it is natural to wonder if it is sufficient to find a candidate measure that does not weight $\mathcal{D}_g$ to get existence. This is precisely the statement of our first main result: if the set of discontinuities of $g$ is given measure $0$ by some candidate measure obtained by a classical asymptotic procedure, then this candidate measure is a stationary compatible measure for $g$.  This result is stronger than any general existence result of the literature that we are aware of, and answers question 1 in an essentially optimal way as we will also explain. We also give several corollaries under more explicit conditions on $g$ and some examples of application. A notable result, for instance, is the fact that if $\mathcal{D}_g$ is countable, and if $g$ is uniformly bounded away from $0$, then existence is granted. We will also provide  a discussion/comparison of several notions of smallness of $\mathcal{D}_g$ (we already mentioned three, which are negative topological pressure, null measure under some candidate measure, and enumerability). 

\vspace{0.3cm}

The second part of the paper focuses on a specific class of $g$-functions, which are called \emph{probabilistic context trees}. In words, a probabilistic context tree is a $g$-function having the property that, for some pasts, we only need to look at a finite suffix of the past in order to obtain the transition probability from the past to  symbols of $A$, and the size of the size of this suffix  is a deterministic function of the past. The set of these suffixes can be represented as a rooted tree that we call \emph{context tree}, which encodes the dependence on the past of the $g$-function. When there exists a stationary stochastic process specified by this $g$-function, we say it is a \emph{variable length memory chain}. Probabilistic context trees are the archetypical examples of locally continuous $g$-functions, because continuity is only assumed along pasts having finite size contexts. 

This class of processes was initially introduced in the information theory literature by \cite{rissanen/1983} as an efficient compression model. Since then, it has been essentially used in mathematical statistics papers as a flexible class of stochastic processes to model real life datas  \citep[for instance]{galves/galves/garcia/garcia/leonardi/2012, belloni/imbuzeiro/2017, cai2017inferring}. On the theoretical side, \cite{gallo/2009} seems to be the first work interested in this model under the perspective of stochastic processes. This class has been also used as a tool to study general $g$-measures, for instance in \cite{gallo/garcia/2013, gallo/paccaut/2013, garivier/2015, imbuzeiro/2015}. Finally,  some recent works have explored the relation with dynamical systems \citep{CCPP} or used this class to create interesting random walk models \citep{lenyeta/2018,lenyetal/2019}.

In the class of variable length memory chains, we obtain results ranging from existence to uniqueness and mixing properties (precisely, $\beta$-mixing, or weak-Bernoullicity). All these results are stated under assumptions on the length of the contexts,  meaning that these results are specially designed for variable length memory chains. Indeed, under our assumptions, only the set of pasts having finite contexts are assumed continuous, and in particular, there is no need for uniform continuity to get existence nor uniqueness. 

\vspace{0.3cm}

The paper is organised as follows. In the next section we introduce the necessary notation and definitions, explain the existence problem and state our first main existence results. In Section \ref{sec:vlmc} we apply the results of Section \ref{sec:main} to the case of variable length memory chains, and present also some results on uniqueness and mixing properties in this context. Section \ref{sec:examples} contains explicit examples (as well as a counterexample), and we provide an interesting discussion of further related questions in Section \ref{sec:discussion}. We conclude giving the proofs of the results in Section \ref{sec:proofs} and an appendix containing proofs of combinatorial results that are used in our discussions.

\section{$g$-functions and existence of $g$-measures}\label{sec:main}

\subsection{Notation} Consider the measurable space $(A,\A)$ where $A$ is a finite set (the alphabet) and $\A$ is the associated discrete $\sigma$-algebra. In this paper we use the convention that $\mathbb{N}=\{0,1,2,\ldots\}$. The sets of left-infinite sequences (pasts)  and of bi-infinite sequences are respectively denoted by $\mathcal{X}^-=A^{-\NN^\star}$  and ${\mathcal{X}}=A^\ZZ$ and will be endowed with the products of discrete topology. More precisely, for any integers $-\infty\leqslant m\leqslant n\leqslant +\infty$, let $\pi^{[m,n]}$ denote the coordinate application taking $x=\ldots x_{-1}x_0x_1\ldots\in\mathcal X$ to  $\pi^{[m,n]}(x)=x_m\ldots x_n\in A^{n-m+1}$. We will use  the notation $\mathcal{F}^{[m,n]}=\sigma(\pi^{[m,n]})$ for the $\sigma$-algebra generated by the coordinate applications $\pi^{[m,n]}$. 
  In particular, $\mathcal{F}^{[m,n]}$ will be denoted $\mathcal{F}^{\leqslant n}$ if $m=-\infty$, $\mathcal{F}^{\geqslant m}$ if $n=+\infty$ and simply $\mathcal{F}$ if $m=-\infty$ and $n=+\infty$. A case that will be more important for us in the sequel is that of  $\mathcal{F}^{\leqslant -1}$ to which we also give the simpler notation $\mathcal F^-$. We will use the shorthand notation $x_m^n=x_m\ldots x_n$ for any integers $-\infty\leqslant m\leqslant n\leqslant +\infty$.  For any pasts $\underline x:=x_{-\infty}^{-1}\in \mathcal{X}^-$ and $\underline z\in \mathcal{X}^-$ and any $k\geqslant0$,  $\underline zx_{-k}^{-1}=\ldots z_{-2}z_{-1}x_{-k}\ldots x_{-1}$ is the concatenation between $x_{-k}^{-1}$ and $\underline z$.
In other words, $\underline zx_{-k}^{-1}$ denotes a new sequence $\underline y\in \mathcal{X}^-$ defined by $y_{i}=z_{i+k}$ for any $i\leqslant -k-1$ and $y_{i}=x_{i}$ for any $-k\leqslant i\leqslant -1$. Finally, we denote by $|v|$ the number of symbols in the finite string $v$, that is, its length.

Denote by ${\M}$ the set of Borelian probability measures on $({\mathcal{X}},\mathcal F)$.
The shift map $T$ on ${\mathcal{X}}$ is defined by $(T{x})_n=x_{n+1}$. A measure $\mu\in\M$ satisfying $\mu=\mu\circ T^{-1}$ is called shift-invariant, or simply invariant or stationary. We denote by $\M_T$ the set of stationary measures in $\M$.

For any finite string $a_m^n$, the cylinder set is denoted $[a_m^n]:=\{x\in\mathcal X:x_m^n=a_m^n\}$ for any $-\infty<m\leqslant n<+\infty$. For any measure $\mu\in\mathcal M$, we will abuse notation by writing $\mu(a_m^n)$ instead of $\mu([a_m^n])$. We may also want to locate a string without making it explicit through its indexes. For instance,  for any  finite string of symbols $v\in A^n,n\geqslant1$, we will use the notation $[v]_i:=\{x\in\mathcal X:x^i_{i-n+1}=v^i_{i-n+1}\}$ in order to locate it. Naturally, if $\mu\in\mathcal{M}_T$, then $\mu([v]_i)$ does not depend on $i$ and we may simply write $\mu(v)$.

\subsection{$g$-functions} The central object of the present work is the \emph{$g$-function}. A $g$-function is  a probability kernel on $A$, specifying the probabilities of transition from infinite pasts to symbols of $A$, just as the extension of a transition matrix for Markov chains:  $g:\mathcal{X}^{\le0}\to [0,1]$ is $\F^{\le0}$-measurable and
\begin{equation}\label{g-function}
\forall \underline x\in \mathcal{X}^{-},\,\,\, \sum_{x_0\in A}g(\underline xx_0)=1.
\end{equation}

\begin{Not} \begin{itemize}\item We will abuse notation by allowing $g$ to act on $\mathcal{X}^{\leqslant i}$ for any $i\in \mathbb Z$ as follows: for any $x\in \mathcal{X}$, $g(x_{-\infty}^{i})=g((T^{i}x)_{-\infty}^{0})$, which entails, in particular, homogeneity. 
\item For any $k\in\mathbb Z,n\geqslant1$ any $a_{k+1}^{k+n}\in A^n$ and any $x_{-\infty}^{k}$, we use the notation
\[g_n(x_{-\infty}^{k}a_{k+1}^{k+n}):=\prod_{i=1}^{n}g(x_{-\infty}^{k}a_{k+1}^{k+i}).\]
\end{itemize}\end{Not}

As simple examples, $g$ is just a probability distribution on $A$ if it only depends on the last coordinate of its argument, and it is a $k$-steps Markov transition matrix if it depends only on the $k+1$ last coordinates.

The role of $g$ will be to describe the conditional expectation of present given past for measures on $\mathcal{X}$, just as a transition matrix describes the conditional probabilities of the present given the last step for a Markov chain. However, unlike in the finite memory case (that is, unlike the case of $k$-step transition matrix with $k<\infty$), it is not obvious that a given $g$-function describes the conditional expectation of some stationary measure. This is the \emph{existence problem}, that we now address. 
%
%

\subsection{The existence problem}
Given a $g$-function $g$, a stationary measure $\mu$ is said to be compatible with (or specified by) $g$ if
\begin{equation}\label{eq:compatible}
{\EE}_{\mu}(\1_{[a]_0}|\F^{-})({x})=g(\underline xa)
\end{equation}
for $\mu$ almost every $x$ and for every $a\in A$. Let $\M_T(g)$ denote the set of stationary measures compatible with $g$. 
The \emph{existence problem} is to find the weakest conditions on $g$ ensuring that $\M_T(g)\neq\emptyset$. 

The main assumption of the literature to solve this issue is that $g$ is continuous at \emph{every} past $\underline x$. We say that $g$ is continuous at  $\underline x$ if for any $a\in A$ and any sequence $\{\underline x^{(k)}\}_{k\geqslant1}$ in $\mathcal{X}^-$ satisfying  $\underline x^{(k)}\rightarrow \underline x$ (with respect to the discrete topology), we have
\[
g(\underline x^{(k)}a)\stackrel{k\rightarrow\infty}{\longrightarrow} g(\underline xa).
\] 
Another traditional assumption of the literature is that $\inf g>0$ (strong positivity), which, combined with continuity, leads to what is sometimes called ``regular $g$-function''. 

The first main objective of the present paper is to solve the existence problem in the case of non-regular $g$-functions. It is worth mentioning here that the expression ``non-regular $g$-measure'', used in the title of the paper, is an abuse of terminology: what is non-regular here is the $g$-function. This abuse nevertheless involves some subtleties that we will discuss in Section \ref{sec:discussion}. {In any case, it is important to emphasise that strong positivity alone is not enough to grant existence (see the example of Subsection \ref{sec:non-exist-positive}), and it is necessary to give conditions on the set $\mathcal{D}_{g}\subset \mathcal{X}^-$ of discontinuities of $g$ in order to guarantee existence. }

 Theorem \ref{theo:exist} below  gives a new and very natural condition on  $\mathcal{D}_{g}$ ensuring that $\mathcal{M}_T(g)\neq\emptyset$. {This result does not assume $\inf g>0$, that is, applies to truly non-regular $g$-functions.}
%
%
%

\subsection{A general existence result}

Let us first construct a natural candidate measure by an asymptotic procedure, using compactness. By Ionescu-Tulcea's Theorem, we know that for any $\underline x\in \mathcal{X}^-$, there exists a unique measure  $\mu^{\underline{x}}\in \M$ satisfying $\mu^{\underline{x}}(C)=\1_{C}(\underline x)$ for any $C\in \F^-$ and, for any $n\geqslant1$ and any $a_0^{n-1}\in A^n$
\begin{equation}\label{eq:ionescu}
\mu^{\underline{x}}(a_{0}^{n-1})=g_n(\underline xa_0^{n-1}).
\end{equation}
In words, this is the measure started with $\underline x$ on $(-\infty-1]$ (Dirac measure on $\underline{x}\in\mathcal{X}^-$), and which we construct iteratively forward (on $[0,+\infty)$) using $g$. Then, let $\mu^{\underline{x},0}:=\mu^{\underline{x}}$ and $\mu^{\underline{x},-i}:=\mu^{\underline{x}}\circ T^{-i}$ for any $i\geqslant1$. 
By compactness of $\M$ (under the topology of weak convergence), the sequence of measures 
\begin{equation}\label{eq:def_barmu}
\bar\mu^{\underline{x},-k}:=\frac{1}{k}\sum_{i=0}^{k-1}\mu^{\underline{x},-i},k\geqslant1
\end{equation}
converges by subsequence. Let $\bar\mu^{\underline{x}}$ be a limit. We have the following theorem. 
{\begin{Theo}\label{theo:exist}
Suppose   that  there exists an $\underline x$ and a limit $\bar\mu^{\underline{x}}$ such that  $\bar\mu^{\underline{x}}(\mathcal{D}_g)=0$. Then $\bar\mu^{\underline{x}}\in\M_T(g)$ which is therefore non-empty.
\end{Theo}}
Formally, when we write $\bar\mu^{\underline{x}}(\mathcal{D}_g)$ we mean $\bar\mu^{\underline{x}}(\{y\in\mathcal X:\underline y\in\mathcal{D}_g\})$. 

All the results that we know about existence  imply that no compatible measure would charge the discontinuous pasts (although we know that this needs not to be the case in general, see Subsection \ref{sec:every} for a counterexample). Here, it is the converse: if the candidate measure does not charge the discontinuous pasts, then it is stationary compatible. 

It is worth mentioning that this result is tight in the following sense: there exists examples of $g$-functions for which no stationary measures exist if we relax the assumption (this is the case of the ``infinite comb'' considered in \cite{CCPP}). This example is presented in Subsection \ref{ex:1}.

\vspace{0.2cm}

Naturally, the main drawback of Theorem \ref{theo:exist} is that its assumption is not easy to check because  the measure $\bar\mu^{\underline{x}}$   is defined through a  limiting procedure. The next subsection focuses on obtaining explicit conditions that are checkable directly on $g$.  


\subsection{Explicit results: how to control the ``smallness'' of $\mathcal{D}_g$}\label{sec:expli}

Theorem \ref{theo:exist} involves the measure $\bar\mu^{\underline{x}}$ which is difficult to access in general. So we seek for topological assumption on the set  $\mathcal{D}_g$, that is, ways to characterise the ``smallness'' of this set without any reference to its measure. 

A first natural way to say that a subset of $\mathcal{X}^{-}$ is small is to say that it is finite, or countable. 

\begin{Cor}\label{cor1}
Suppose that $g\geqslant\epsilon>0$ and that $\mathcal{D}_g$ is countable, then $\M_T(g)\neq\emptyset$.
\end{Cor}

We will explain later  that enumerability of $\mathcal{D}_g$  is far from necessary for existence. 

Let us now mention another way to characterise the smallness of $\mathcal{D}_g$ which was introduced by  \cite{gallo/paccaut/2013}. First denote by $\mathcal{D}^n_g$ the set of strings of size $n$ which are prefix of discontinuous pasts of $g$, that is, 
\[
\mathcal{D}^n_g:=\{x_{-n}^{-1}:\underline x\in\mathcal{D}_g\}.
\]
Now, we define the topological pressure of  $\mathcal{D}_g$ as 
\begin{equation}\label{eq:top_entropy}
P_g(\mathcal{D}_g)=\limsup_{n\to+\infty}\frac1n\log\sum_{a_{-n}^{-1}\in\mathcal{D}_g^n}\sup_{\underline{x}}g_n(\underline x a_{-n}^{-1}).
\end{equation}
\cite[Theorem 1]{gallo/paccaut/2013} states that there exists at least one stationary measure compatible with $g$  if $P_g(\mathcal{D}_g)<0$ together with some  positivity condition on $g$ (that we do not define here for the sake of presentation). The next proposition states that $P_g(\mathcal{D}_g)<0$, alone, implies the condition of our Theorem \ref{theo:exist}, showing in particular that our results are strictly stronger than those of \cite{gallo/paccaut/2013}. 

{\begin{Cor}\label{prop1}
Suppose that $P_g(\mathcal{D}_g)<0$, then there exists a past $\underline{x}$ and a limit $\bar\mu^{\underline{x}}$ such that $\bar\mu^{\underline{x}}(\mathcal{D}_g)=0$. 
\end{Cor}}

The assumption $P_g(\mathcal{D}_g)<0$  does not involve any measure, but it strongly relies on detailed properties of $g$. 
For this reason, \cite{gallo/paccaut/2013} also considered the  \emph{upper exponential growth} of   $\mathcal{D}_g$. Since we are assuming that  $\mathcal{D}_g$ is infinite, we know that the cardinal $|\mathcal{D}^n_g|$ diverges as $n$ diverges. How fast it diverges can be measured by the \emph{upper exponential growth}
\begin{equation}\label{eq:growth}
\bar{gr}(\mathcal{D}_g):=\limsup_{n}|\mathcal{D}^n_g|^{1/n}.
\end{equation}
This  is yet  another way to characterise how large is $\mathcal D_g$ in $\mathcal{X}^-$. It was proved in  \cite{gallo/paccaut/2013} that, for $g\geqslant\epsilon$, 
\begin{equation}\label{eq:growth_assumption}
\bar{gr}({\mathcal{D}_g})<[1-(|A|-1)\epsilon]^{-1}
\end{equation} implies that $\mathcal{D}_g$ has strictly negative topological pressure, and therefore was  sufficient  for existence. According to Corollary \ref{prop1} this is also a sufficient condition for our criterium to hold. 

\subsection{A remark on the cardinality of $\mathcal{D}_g$}\label{sec:preceding} It is interesting to notice that the  growth (that is how fast $|\mathcal{D}^n_g|$ diverges) is not trivially related to the cardinality of the set. It  is easy for instance, for any diverging natural function $f(n),n\in\mathbb{N}^\star$ (no matter how slow it diverges), to build an uncountable set $\mathcal{D}_g$ with $|\mathcal{D}^n_g|=f(n)$ (see Lemma \ref{lemma:tree1}). So assumption \eqref{eq:growth_assumption} already covers examples of $g$'s having uncountably many discontinuities. 

On the other hand,  it is also possible to construct countable sets with  rapidly diverging $|\mathcal{D}^n_g|$, giving examples of results covered by Corollary \ref{cor1} which are not covered by assumption \eqref{eq:growth_assumption} (which was the only explicit assumption of \cite{gallo/paccaut/2013}). For instance, in the case of binary alphabets, assumption \eqref{eq:growth_assumption} imposes  $\bar{gr}({\mathcal{D}_g})<2$. But for any $f(n)=o(2^n)$ we can construct a countable set $\mathcal{D}_g$ with $|\mathcal{D}^n_g|=f(n)$ (see Lemma \ref{lemma:tree2}). So taking for instance $f(n)=\frac{2^n}{\lfloor \log n\rfloor}$ we have $\bar{gr}({\mathcal{D}_g})=\limsup_n\frac{2}{\lfloor \log n\rfloor^{1/n}}=2$ which therefore does not satisfy assumption \eqref{eq:growth_assumption} but satisfies the assumption of Corollary \ref{cor1}

\subsection{Explicit results: using the form of $\mathcal{D}_g$} To conclude concerning the explicit results on  existence, let us also mention the following assumption, which is not on the size of the set $\mathcal{D}_g$, but rather on its form.  Suppose that there exists a finite string $v$ that does not appear as substring in the elements of $\mathcal{D}_g$, that is, for any $\underline x\in\mathcal{D}_g$, we have $x_i^{i+|v|-1}\ne v$ for any $i\leqslant -|v|$. We then say that $\mathcal{D}_g$ is \emph{$v$-free}. 

\begin{Cor}\label{cor2}
Consider a $g$-function $g$ for which $\mathcal{D}_g$ is $v$-free for some finite string $v$, and assume that $\inf_{\underline x}g_{|v|}(\underline x v)>0$, then $\M_T(g)\neq\emptyset$. In particular, the positivity condition is satisfied under strong positivity $\inf g>0$.
\end{Cor}

An example of application of Corollary \ref{cor2} is $\mathcal{D}_g=\{1,3\}^{-\mathbb N}$ on the $3$-letters alphabet $A=\{1,2,3\}$. Notably, existence was not shown at this level of generality for this example in the literature.  \cite[see Corollary 1 and Example 4 therein]{gallo/paccaut/2013} need the parameter $\epsilon$ to be sufficiently large to compensate the exponential growth of the set of discontinuities. This is also the case of \cite[]{gallo/garcia/2018} which relies on the results of \cite{gallo/paccaut/2013}. \cite{gallo/2009} and  \cite{gallo/garcia/2013} on the other hand obtained existence adding a condition on the continuity set $\mathcal{X}^-\setminus\mathcal{D}_g$  which is rather strong and in particular implies also uniqueness. Using our Corollary \ref{cor2}, we guarantee existence without any further condition than  $v$-freeness. 

\subsection{Uniqueness} As we already mentioned in introduction, to get uniqueness usually requires some control on the \emph{variation} of $g$.  The variation of order $l$ of  $g$ at $\underline y$ is defined by
\[
\text{var}^g_l(\underline{y}):=\sup_{a\in A}\sup_{z:z_{-l}^{-1}=y^{-1}_{-l}}|g(\underline za)-g(\underline{y}a)|.
\]
Clearly, $\underline y$ is a continuity point for $g$ if and only if $\text{var}^g_l(\underline{y})$ vanishes as $l$ diverges. We do not extend much here because our objective will be to focus on the variable length case in the next section, but let us just mention  that, under the condition of Theorem \ref{theo:exist}, if we further assume that $\mathbb E_{\bar \mu^{\underline x}}\sum_l\left(\text{var}^g_l\right)^2<\infty$, then uniqueness is granted. This fact follows immediately from a result in \cite{johansson/oberg/2003}.

\section{Specialising to variable length memory chains}\label{sec:vlmc}

In the present section we consider a particular class of $g$-measures which we call \emph{variable length memory chains} that we now define. 
Given a $g$-function $g$, for any $\underline x$, let us define 
\begin{align*}
\ell^g(\underline x)&:=\inf\{k\geqslant1:g(\underline yx_{-k}^{-1}a)=g(\underline zx_{-k}^{-1}a)\,,\,\,\forall a,\underline y,\underline z\}
\end{align*}
(with the convention that $\ell^g(\underline x)=\infty$ if the set is empty), which is equivalent to $\ell^g(\underline x):=\inf\{k\geqslant1:\text{var}_k^g(\underline x)=0\}$. Call $x_{-\ell^g(\underline x)}^{-1}$ the context of $\underline x$, it is the smallest suffix of $\underline x$ we need to get the distribution of the next symbol according to $g$. The set $\tau^g:=\cup_{\underline x}\{x_{-\ell^g(\underline x)}^{-1}\}$ is sometimes called context tree, because it can be pictorially represented as a rooted tree in which each path from the root to a leaf represents a context.  Some references impose $\tau^g$ to be countable or finite, this is usual in the statistics literature but not only \citep{CCPP,cenacetal/2018} (naturally the set of finite contexts is always countable, but $\tau^g$ itself needs not be countable). We do not assume this here.

Since by definition $g(\underline xa)$ does not depend on $x_{-\infty}^{-\ell^g(\underline x)-1}$, we  define the probability kernel $p^g:A\times \mathcal{X}^{-}\cup \bigcup_{i\geqslant0}A^{\{-i,\ldots,-1\}}\rightarrow[0,1]$ through
\[
p^g(a|x_{-\ell^g(\underline x)}^{-1}):=g(\underline xa)\,\,,\,\,\,\forall a\in A, \underline x\in\mathcal{X}^-. 
\]
We can now identify $g$ with the pair $(\tau^g,p^g)$, which we call  \emph{a probabilistic context tree}. A stationary measure $\mu$ compatible, in the sense of \eqref{eq:compatible}, with a probabilistic context tree is call a \emph{variable length memory chain} (VLMC, not to be confused with Variable Length Markov Chains, which are assumed to have a \emph{finite} context tree). Observe that, for such  measures, \eqref{eq:compatible} now reads, for any $x$ such that $\ell(\underline x)<\infty$
\begin{equation}\label{eq:compatibleVLMC}
{\EE}_{\mu}(\1_{[a]_0}|\F^{-})({x})=\mu([a]_0|x_{-\ell^g(\underline x)}^{-1})=p^g(a|x_{-\ell^g(\underline x)}^{-1})
\end{equation}
in which 
\[
\mu([a]_0|x_{-\ell^g(\underline x)}^{-1}):=\frac{\mu(x_{-\ell^g(\underline x)}^{-1}a)}{\mu(x_{-\ell^g(\underline x)}^{-1})}.
\]

\begin{Not}\begin{itemize}\item Sometimes, in order to avoid overloaded notation, we will avoid the reference to $g$ and simply write $\ell,\tau,p$, keeping in mind that they all refer to a specific $g$-function $g$. 
\item Just as $g$ has been extended to act on $\mathcal{X}^{\leqslant i}$ for $i\in \mathbb Z$, we also extend $\ell$ to take as argument any sequence $x_{-\infty}^{i}$, $i\in\mathbb Z$, by $\ell(x_{-\infty}^{i})=\ell((T^{i+1}x)_{-\infty}^{-1})$. 
\item The set $\{\underline x:\ell(\underline x)>k\}$ is measurable with respect to $\mathcal{F}^{[-k,-1]}$.  Thus, using the above notation extension, we will make sense out of the set $\{x\in\mathcal{X}:\ell(x_{i+1}^{i+j})>j\}$ (or $=$ or  $<$). 
\end{itemize}
\end{Not}

\begin{Theo}\label{exist_uniq_VLMC}
If there exists $\underline x$ such that $\bar\mu^{\underline{x}}(\ell^g<\infty)=1$, then  $\bar\mu^{\underline{x}}\in\mathcal{M}_T(g)$.  
Under the stronger assumption that $\EE_{\bar\mu^{\underline{x}}}\ell^g<\infty$, and adding that $\inf g>0$, the set $\mathcal{M}_T$ has a unique element. 
\end{Theo}

In other words, if the length of the contexts ``seen'' by the candidate measure is a.s. finite, then the candidate is actually a compatible VLMC. This amounts to say that the corresponding stationary process goes from finite size contexts to finite size contexts, that is, lives on a countable set of contexts. If moreover this length has finite expectation, then this compatible measure is unique (modulo strong positivity $g$). This is quite a natural statement, and as far as we know, the only one involving the length of the context function as main data. Previous results of the literature, such as \cite{gallo/2009,gallo/garcia/2013} mainly focussed, as far as context trees are concerned, on \emph{$v$-free context trees} (those trees for which the set of infinite contexts is $v$-free). The idea to work on the form of $\tau$ was pushed forward by \cite{cenacetal/2018}, in which more refined combinatorial properties of $\tau$ are used to obtain sufficient and necessary conditions for existence and uniqueness of the compatible measure. The comparison with their results would be interesting but is not straightforward. We choose not to enter into the details here.
\vspace{0.2cm}

Theorem \ref{exist_uniq_VLMC} is nevertheless quite inexplicit, as Theorem \ref{theo:exist} was, because the assumption is based on $\bar\mu^{\underline{x}}$. So we now focus on obtaining ``easier to check" results. In this direction, the next result will involve $\mu^{\underline{x}}$ which is way more simple to understand than $\bar\mu^{\underline{x}}$. 

{For any given stationary measure $\mu\in\mathcal M_T$, let us define  
\[
\beta(n):=
\sup\frac{1}{2}\sum_{i=1}^{I}\sum_{j=1}^{J} \left|\mu(A_i\cap
B_j)-\mu(A_i)\mu(B_j)\right|
\]
in which the supremum is taken over all the possible partitions $\{A_1,\ldots,A_I\}$ and $\{B_1,\ldots,B_J\}$ of $\XX$ satisfying $A_i\in\mathcal{F}^{\le0}$ and $B_j\in\mathcal{F}^{\geqslant n}$.  Then, we say \citep{bradley/2005} that $\mu$  is \emph{weak Bernoulli} (or, equivalently, \emph{absolute regular} or \emph{$\beta$-mixing}) if $\beta(n)$ vanishes as $n$ diverges.}
The reader should keep in mind that this is much stronger than having a mere ``mixing'', which corresponds to the absolute value above vanishing as $n$ diverges for any pair $A_i,B_j$. Anyway, what matters is not so much the definition itself, but rather its implications. \cite{shields/1996} lists several implications of weak-Bernoullicity. Here, let us cite only two    interesting consequences of a given stationary measure $\mu$ being weak-Bernoulli:
\begin{itemize}
\item $\mu$ can be obtained as a coding factor of some product measure $\nu$, that is, there exists a function $f:A^{\mathbb Z}\rightarrow B^\mathbb Z$ ($B$ some other finite aphabet) that commutes with the shift $T$, and such that $\mu=\nu\circ f^{-1}$.
\item The double tail $\sigma$-algebra $\mathcal T:=\cap_{n\geqslant1}\left(\mathcal{F}^{\leqslant -n}\vee\mathcal F^{\geqslant n}\right)$ is trivial under $\mu$.
\end{itemize} 

Based on the measure $\mu^{\underline x}$ (defined through \eqref{eq:ionescu}), the next theorem gives a sufficient condition for the existence of a unique compatible measure, which moreover is weak-Bernoulli. 

\begin{Theo}\label{theo:unifo}
Suppose that there exists $\underline x\in \mathcal{X}$ such that $\sup_{\underline x}\mu^{\underline{x}}(\{y_0^{+\infty}:\ell^g(y_0^{n-1})>n\})\stackrel{n}{\rightarrow}0$, then  $\bar\mu^{\underline{x}}\in\mathcal{M}_T(g)$.  
If moreover we have $g>0$ and  for any $\underline x$
\begin{equation}\label{eq:assumption-iii-us}
\mu^{\underline{x}}\left(\left\{y_0^{+\infty}:\sum_{n\geqslant1}\1_{(n,\infty)}(\ell^g(y_0^{n-1}))<\infty\right\}\right)=1,
\end{equation}
 then there exists a unique compatible measure and it is weak-Bernoulli.  
 \end{Theo}

We  conclude with an even more explicit statement, involving the growth of $\tau$ that we now define. Let $\tau^n:=\{y_{-n}^{-1}:\ell(\underline y)> n\}$ be the set of strings of size $n$ that are suffix of contexts of $\tau$. The \emph{upper exponential growth of $\tau$} is
\begin{equation}\label{eq:growth_tree}
\bar{gr}(\tau):=\limsup_{n}|\tau^n|^{1/n}.
\end{equation}

\begin{Cor}\label{prop:explicitVLMC}
Suppose that $g\geqslant\epsilon$. If $ \bar{gr}({\tau^g})<[1-(|A|-1)\varepsilon]^{-1}$ then there exists a unique $g$-measure which is weak-Bernoulli. 
\end{Cor}

As mentioned earlier (see \eqref{eq:growth_assumption} above), a similar criterium (not exactly the same because here we consider the growth of the tree, and not only of the set of discontinuities) was proven to be an explicit sufficient condition for existence in \cite{gallo/paccaut/2013}. Here  we improve this to uniqueness and weak-Bernoullicity. 

\section{Explicit examples}\label{sec:examples}

{This section contains several examples of applications of the results of the preceding sections. The last subsection, on the contrary, contains a counterexample showing how things can go so wrong that the discontinuities prevent any stationary measure to be compatible, even though the $g$-function is strongly positive.}

%
\subsection{The renewal process}\label{ex:1}
This example is a well known example, even under the perspective of testing existence questions \citep{CCPP}. The main objective here is to show that Theorem \ref{theo:exist} is tight in the sense that we cannot relax its assumption in general. For the renewal processes indeed, our theorem retrieves the exact (necessary and sufficient) condition for existence as we now explain.

Let $q=\{q_i\}_{i\in\NN^\star}$ be a sequence of $(0,1)$-valued real numbers and define $\ell_1(\underline{x}):=\inf\{k\geqslant1:x_{-k}=1\}$ (with $\ell_1(\underline{0})=\infty$ in which $\underline 0$ denotes the past composed of infinitely many $0$'s). Consider the $g$-function on $A=\{0,1\}$ defined through $g(\underline{x}1):=q_{\ell_1(\underline{x})}$, in which we also fixed some $q_\infty\in[0,1]$. This model falls in the class of probabilistic context trees and in particular we have $\text{var}_k^g(\underline x)=0$ for any $k\geqslant \ell_1(\underline{x})$ when $\underline{x}\neq\underline{0}$, while
\[
\text{var}_k^g(\underline{0})=\sup_{m,n\geqslant k}|q_m-q_n|.
\]

We distinguish 3 cases, according to $V(q):=\sum_{k\geqslant1}\prod_{i=1}^{k}(1-q_i)$ and $q_\infty$.
\begin{enumerate}
\item Case 1. Suppose that $V(q)<\infty$. Choosing any $\underline{x}$ with $x_{-1}=1$, we notice that $\mu^{\underline{x}}$ is the measure of an undelayed renewal sequence \citep[Chapter 4, Section 3.2]{bremaud2013markov}, and the renewal theorem guarantees that $\mu^{\underline{x},-i}([1]_0)\rightarrow 1/m(q)$ where $m(q)$ is the expected distance between two consecutive $1$'s in the sequence. It is not difficult to see that this expectation is (recall $x_{-1}=1$)
\begin{align*}
m(q)&=\sum_{k\ge0}k\mu^{\underline x}([0^k1]_{k})=1+
\sum_{k\geqslant1}\mu^{\underline{x}}([0^k]_{k-1})\\&=1+\sum_{k\geqslant1}(1-q_1)(1-q_2)\ldots(1-q_{k})\\&=1+V(q)
\end{align*}
where $0^k$ stands for the string of $k$ consecutive $0$'s.
In other words, if $V(q)<\infty$, we have $\mu^{\underline{x},-i}([1]_0)\rightarrow 1/m(q)>0$, and therefore $\bar\mu^{\underline x}(1)>0$ since it is the Ces\`aro mean limit of $\mu^{\underline{x},-i}([1]_0)$'s. Since it is a stationary measure, this automatically implies that   $\bar\mu^{\underline{x}}(\underline0)=0$, and therefore that the condition of our theorem also holds, since $\underline{0}$ is the only possible discontinuity point. We therefore have existence in this case.

\item Case 2. Suppose that $V(q)=\infty$ and $q_\infty=0$. This assumption implies that $q_i\rightarrow0=q_\infty$ and therefore we have continuity everywhere. We therefore also have existence in this case. Using the same argument as above we can also see that, since $V(q)=\infty$  implies that the expected length between two consecutive $1$'s is infinite, the renewal theorem  applied to the undelayed renewal sequence $\mu^{\underline x}$ (with $x_{-1}=1$) gives $\mu^{\underline{x},-i}([1]_0)\rightarrow``1/\infty"=0$. Thus, by the same argument as above, $\bar\mu^{\underline x}(1)=0$ and thus, by stationarity, $\bar\mu^{\underline x}$ is the Dirac measure on the point ``all $0$''.

\item Case 3. Suppose that $V(q)=\infty$ and $q_{\infty}>0$. This assumption implies that $q_i\rightarrow0\ne q_\infty>0$, and therefore, that $\underline0$ is a discontinuity point (and indeed it is the only one). We will prove that $\bar\mu^{\underline x}(\underline 0)>0$ for any $\underline x$, that is, that this case violates the conditions of Theorem  \ref{theo:exist}. Since this  is precisely the case of non-existence, as proved in \cite{CCPP}, this shows that our theorem is optimal on this example. 

First,  $V(q)=\infty$  implies that the expected length between two consecutive $1$'s is infinite. For any $\underline x$ such that $\ell_1(\underline x)<\infty$, $\mu^{\underline{x},-i}$ is a renewal sequence, with delay $\ell_1(\underline x)-1$ if $\ell_1(\underline x)>1$ \citep[Chapter 4, Section 3.2]{bremaud2013markov}. The renewal  theorem now gives $\mu^{\underline{x},-i}([1]_0)\rightarrow``1/\infty"=0$ for any such $\underline x$. It is not difficult to see that the same convergence holds for the remaining past $\underline0$, since assuming $q_{\infty}>0$ implies that a $1$ occurs with probability 1 under $\mu^{\underline0}$, taking us back to the undelayed case. 
The convergence to $0$ for any $\underline x$ automatically implies that $\bar\mu^{\underline x}([1]_0)=0$, that is, $\bar\mu^{\underline x}(\underline 0)=1$. In other words, the only discontinuity has $\bar\mu^{\underline x}$ full measure for any $\underline x$, and the condition of Theorem  \ref{theo:exist} is violated.

\end{enumerate}

An extension of this renewal argument could be done for $v$-free probabilistic context trees $(\tau,p)$ in general, that is, when $v$ does not appear as substring of the infinite size contexts of $\tau$. For instance, the context tree of the renewal process is $1$-free. Reasoning as above, if we are able to show that, for some past $\underline x$, $\mu^{\underline x,-i}([v]_{|v|-1})$ has a strictly positive limit, then $\bar\mu^{\underline x}$ is a stationary compatible measure by Theorem \ref{theo:exist}. Notice that if we start with any $\underline x$ satisfying $x_{-|v|}^{-1}=v$, then $\mu^{\underline x,-i}([v]_{|v|-1}),i\geqslant0$ does not depend on $x_{-\infty}^{-|v|-1}$, so we can actually use the notation $\mu^{v,-i},i\geqslant0$ in this case. Consider 
\[
T^{v}(x_0^{+\infty})=\inf\{t\geqslant|v|-1:x_{t-|v|+1}^{t}=v\}. 
\]
What can be proved in this case, using a renewal argument, is that if $\mathbb E_{\mu^{v}}T^{v}<\infty$ then $\mu^{v,-i}([v]_{|v|-1})$ has strictly positive limit, which would guarantee existence. Observe that the condition $\mathbb E_{\mu^{v}}T^{v}<\infty$ actually only depends on the set of transition probabilities $p$ of the pair $(\tau,p)$. We do not get into such details, but it would be interesting to investigate the relation between this assumption ($v$-free probabilistic context tree with finite expected return to $v$) with the assumptions of \cite{cenacetal/2018} involving finiteness of certain series and the notions of ``$\alpha$-lis" and ``stability'' therein defined.

\subsection{Existence for $1$-free cases on $A=\{0,1\}$}\label{ex:ex_vfree}

The binary renewal process introduced above is a very special $1$-free example, in which the distance to the last occurrence of $1$ backwards in time is the size of the context. 
We have proved  that, when $q_\infty>0$, there exists a stationary measure if and only if $V(q)$ is finite. But the main general characterisation of existence was that $\bar\mu^{\underline x}$ gives positive weight to the symbol $1$. This latter is also a sufficient and necessary condition for any $1$-free examples, that is, for any $g$ having $\underline 0$ as  unique discontinuity point since we are on $A=\{0,1\}$. In general however, it is complicated to get one explicit condition which is at the same time sufficient and necessary. We now explain how to get necessary and also sufficient condition. Let $s=\{s_i\}_{i\geqslant1}$ and $r=\{r_i\}_{i\geqslant1}$ be the $[0,1]$-valued sequences defined through $s_i:=\inf_{\underline x}g(\underline x10^{i-1}1)$ and $r_i:=\sup_{\underline x}g(\underline x10^{i-1}1)$ respectively.  We claim that if $V(s)<\infty$ there exists a stationary measure compatible with $g$, and if $V(r)=\infty$ then there does not exist such measure. Let us explain  the  case $V(s)<\infty$ (the case $V(r)=\infty$ follows equally). Choose any $\underline{x}$ with $x_{-1}=1$ and consider the measure $\mu_{\text{ren}}^{\underline x}$ obtained starting from $\underline x$ and using the probabilistic context tree of the renewal process, as in the preceding subsection, but  with sequence $s$ instead of $q$. If $V(s)<\infty$, we know that $\mu_{\text{ren}}^{\underline x,-k}([1]_0)\rightarrow 1/m(s)>0$. So we are done if we are able to prove that $\mu_{g}^{\underline x,-k}([1]_0)\geqslant \mu_{\text{ren}}^{\underline x,-k}([1]_0)$ for any $k$. This follows from a well-known coupling argument \citep[see Part IV chapter 2]{lindvall/1992}  between ordered kernels (in our case, $g$-functions). Indeed, consider two $g$-functions $g$ and $h$ (on $A=\{0,1\}$ in the present case, but can be extended) which are ordered in the sense that $g(\underline x1)\geqslant h(\underline y1)$ for any $\underline x\geqslant\underline y$ with respect to coordinate-wise order. Then, for any ordered pasts $\underline x\geqslant \underline y$, there exists a  measure $\nu^{(\underline x,\underline y)}$ on $(A\times A)^{\mathbb N}$, which is a coupling of $\mu^{\underline x}_g$ and $\mu^{\underline y}_h$, and which is supported on $\{(x_{0}^{+\infty},y_{0}^{+\infty}):x_i\geqslant y_i,i\geqslant0\}$. This means in particular that $\mu^{\underline x,-i}_g([1]_0)\geqslant\mu^{\underline y,-i}_h([1]_0)$ for all $i\geqslant0$.

\subsection{More examples of the literature}

Examples can be taken from \cite{gallo/2009,gallo/garcia/2013,gallo/paccaut/2013,Desantis/piccioni/2012}. All these works assumed that $\inf g>0$, with exception of \cite{gallo/paccaut/2013}, who assumed something slightly weaker than strong positivity (their assumption is technical and will not be introduced here for the sake of brevity). We claim that our results from Section \ref{sec:main} cover any of these examples as they were presented. We refer the interested reader to these references.

\subsection{Examples of application of Theorem \ref{theo:unifo}}
A first natural application of Theorem \ref{theo:unifo} is the renewal process itself, with sequence $q$ satisfying $q_i\in(0,1)$ for any $i\in\mathbb N^\star\cup\{\infty\}$ and  $V(q)<\infty$ (see Subsection \ref{ex:1}). Indeed,  observe that
\begin{equation}\label{ex:renew_bern}
\mu^{\underline x}(\ell_1(y_0^{i-1})>i)=\mu^{\underline x}([0^i]_{i-1})=\prod_{j=k}^{k+i-1}(1-q_j)
\end{equation}
if $\ell_1(\underline x)=k$, and we also have to consider separately the case $\underline x=\underline 0$ which gives $\mu^{\underline x}([0^i]_{i-1})=(1-q_{\infty})^i$. 
By the assumptions and using Borel-Cantelli we conclude that $\mu^{\underline x}(\ell_1(y_0^{i-1})>i\,,\,\text{i.o.})=0$ for any $\underline x$, which means that Theorem \ref{theo:unifo} applies. So we have uniqueness and weak-bernoullicity of the renewal process but this is well-known in the literature (see \cite{shields/1996} for instance). 

Let us now consider another class, generalising the renewal process, which was introduced by \cite{gallo/2009}. Consider a function $h:\mathbb{N}\rightarrow \mathbb{N}$ and consider the probabilistic context tree $(\tau,p)$ with $\tau=\{\underline 0\}\cup\cup_{i\geqslant1}\cup_{v\in A^{h(i)}}\{v10^{i-1}\}$. There is a unique infinite context, the past $\underline 0$, just as for the renewal process. Here, if the last occurrence of a $1$ when we look backward is at distance $j$, then the size of the context is the deterministic function $H(j):=j+h(j)$ (the renewal process is a particular case with $h\equiv0$). In any case, we are in presence of a $1$-free probabilistic context tree, and we know (see Subsection \ref{ex:ex_vfree}) that, letting $s_i:=\inf_{\underline x}g(\underline x10^{i-1}1),i\geqslant1$, the assumption $V(s)<\infty$ guarantees existence of a stationary measure. We now seek for a sufficient condition on the sequences $s$ and $h$ to get uniqueness and weak-bernoullicity, and we wish to do this without assuming $\inf g >0$ but only $g>0$.  In order to simplify the presentation, suppose that $H$ is strictly increasing,  denote by $H^{\leftarrow}$ its inverse and assume also that $s_i$ decreases to $0$ (but $s_\infty>0$ in order to keep $g>0$). Observe that for any $\underline x$, $\mu^{\underline x}(\ell(y_0^{i-1})>i)=\mu^{\underline x}([0^{H^{\leftarrow}(i)}]_{i-1})$. We will use  the renewal measure $\mu_{\text{ren}}^{\underline x}$ defined in \ref{ex:ex_vfree}. The same ordering argument used there applies here and  gives $\mu^{\underline x}([0^{H^{\leftarrow}(i)}]_{i-1})\le\mu_{\text{ren}}^{\underline x}([0^{H^{\leftarrow}(i)}]_{i-1})$. Observe that
\begin{align}\label{eq:1freeWB}
\mu_{\text{ren}}^{\underline x}([0^{H^{\leftarrow}(i)}]_{i-1})&=\mu_{\text{ren}}^{\underline x}([0^{i}]_{i-1})+\sum_{l=0}^{i-H^{\leftarrow}(i)-1}\mu_{\text{ren}}^{\underline x}([10^l0^{H^{\leftarrow}(i)}]_{i-1}).
\end{align}
We now need to show that both terms are summable in $i$ (for any $\underline x$), since then, using Borel-Cantelli and Theorem \ref{theo:unifo}, we can conclude uniqueness and weak-bernoullicity. Using exactly the same argument as for \eqref{ex:renew_bern}, the first term of \eqref{eq:1freeWB} is summable independently of $\underline x$ and of $h$, and no further assumption on $s$ has to assumed else than $V(s)<\infty$ for this term. We now come to the second term of \eqref{eq:1freeWB}, which equals
\[
\sum_{l=0}^{i-H^{\leftarrow}(i)-1}\mu_{\text{ren}}^{\underline x}([1]_{i-l-H^{\leftarrow}(i)-1})\prod_{j=1}^{l+H^{\leftarrow}(i)}(1-s_j). 
\]
This term is upper bounded (simply use $\mu_{\text{ren}}^{\underline x}([1]_{i-l-H^{\leftarrow}(i)-1})\leqslant1$) by
\[
\sum_{l=0}^{i-H^{\leftarrow}(i)-1}\prod_{j=1}^{l+H^{\leftarrow}(i)}(1-s_j),
\]
independently of $\underline x$. So we will have uniqueness and weak-Bernoullicity if 
\[
\sum_{i\geqslant1}\sum_{l=0}^{i-H^{\leftarrow}(i)-1}\prod_{j=1}^{l+H^{\leftarrow}(i)}(1-s_j)<\infty. 
\]
As an example, let us consider the case in which $s_i=1-\left(\frac{i}{i+1}\right)^\alpha,i\geqslant1$, with $\alpha>1$, so that  $V(s)<\infty$. Let $h(i)=C_1\times i^\delta$ for some positive constants $C_1\in\mathbb N$ and $\delta>0$. Then we have $H^{\leftarrow}(i)=C_2i^{1/\delta}$ for some $C_2>0$ and therefore 
\begin{align*}
\sum_{i\geqslant1}\sum_{l=0}^{i-H^{\leftarrow}(i)-1}\prod_{j=1}^{l+H^{\leftarrow}(i)}(1-s_j)&=\sum_{i\geqslant1}\sum_{l=0}^{i-H^{\leftarrow}(i)-1}\frac{1}{(l+H^{\leftarrow}(i)+1)^\alpha}\\
&=\sum_{i\geqslant1}\sum_{l=H^{\leftarrow}(i)+1}^{i}\frac{1}{l^\alpha}\le\sum_i\sum_{l\geqslant H^{\leftarrow}(i)+1}\frac{1}{l^\alpha}\\
&\le C_3\sum_i\frac{1}{H^{\leftarrow}(i)^{\alpha-1}}=C_4\sum_i\frac{1}{i^{(\alpha-1)/\delta}}
\end{align*}
which is summable if and only if $\delta+1<\alpha$.

\subsection{An example of application of Corollary \ref{prop:explicitVLMC}}
\ \\
\vskip 10pt

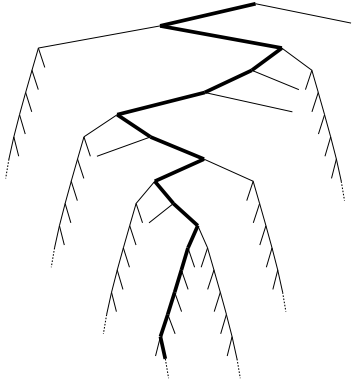
\begin{minipage}{0.42\textwidth}
\begin{tikzpicture}[scale=0.2]
\tikzset{every leaf node/.style={draw=none,circle=none},every internal node/.style={draw,circle,fill,scale=0.01}}
\tikzset{sibling distance=2pt}
\tikzset{level distance=42pt}
\Tree
[.{} \edge[line width=7pt];
[.{} 
	[.{} 
		[.{}
			[.{}
				[.{} 
					[.{} 
						[.{} 
							\edge[line width=2pt,dashed];\node[fill=none,draw=none]{};\edge[draw=none];\node[fill=none,draw=none]{};
						]
						{}
					]
					{}
				]
				{}
			]
			{}
		]
		{}
	] 
\edge[line width=7pt]; [.{} \edge[line width=7pt];
[.{} \edge[line width=7pt];
[.{} \edge[line width=7pt];
[.{} 
[.{} 
[.{}
[.{}
[.{} 
[.{} 
[.{} 
\edge[line width=2pt,dashed];\node[fill=none,draw=none]{};\edge[draw=none];\node[fill=none,draw=none]{};
]
{}
]
{}
]
{}
]
{}
]
{}
] 
\edge[line width=7pt]; [.{} 
{} 
\edge[line width=7pt]; [.{} \edge[line width=7pt];
[.{} 
[.{} 
[.{}
[.{}
[.{} 
[.{} 
[.{} 
\edge[line width=2pt,dashed];\node[fill=none,draw=none]{};\edge[draw=none];\node[fill=none,draw=none]{};
]
{}
]
{}
]
{}
]
{}
]
{}
] 
\edge[line width=7pt]; [.{} 
{} 
\edge[line width=7pt]; [.{} 
\edge[line width=7pt]; [.{}
\edge[line width=7pt]; [.{} 
\edge[line width=7pt]; [.{}
\edge[line width=7pt]; [.{}
\edge[line width=7pt]; [.{}
{} 
\edge[line width=7pt]; [.{} 
\edge[draw=none];\node[fill=none,draw=none]{};\edge[line width=2pt,dashed];\node[fill=none,draw=none]{};
] 
]
{}
]
{}
]
{}
]
{}
]
[.{} 
{}
[.{}
{}
[.{} 
{}
[.{}
{}
[.{}
{}
[.{}
\edge[draw=none];\node[fill=none,draw=none]{};\edge[line width=2pt,dashed];\node[fill=none,draw=none]{};
]
]
]
]
]
] 
]
]
]
[.{} 
{}
[.{}
{}
[.{} 
{}
[.{}
{}
[.{}
{}
[.{}
\edge[draw=none];\node[fill=none,draw=none]{};\edge[line width=2pt,dashed];\node[fill=none,draw=none]{};
]
]
]
]
]
] 
]
] 
] 
{} 
]
{} 
]
[.{} 
{}
[.{}
{}
[.{}
{}
[.{}
{}
[.{}
{}
[.{}
\edge[draw=none];\node[fill=none,draw=none]{};\edge[line width=2pt,dashed];\node[fill=none,draw=none]{};
]
]
]
]
]
] 
] 
] 
{} 
]
\end{tikzpicture}
\end{minipage}%
\begin{minipage}{0.55\textwidth}
We define a probabilistic context tree $g=(\tau,p)$ on  $A=\{0,1\}$ in which every infinite context will be a discontinuity point of $g$. The context tree $\tau$ is built around the infinite context (the \emph{trunk}) obtained by concatenating the binary coding of all the natural numbers (in bold on the picture)  $\ldots0001101100010$. We then add the infinite \emph{branches}  of the form $\underline\alpha w(k)$ where $w(k)$ is the concatenation of the codings of the $k$ first numbers and $\alpha$ is equal to the last symbol of $w(k)$ (that is, formally, $\alpha=w(k)_{-|w(k)|}$). The \emph{leaves} of these branches are of the form $\beta\alpha^l w(k)$ where $\beta=1-\alpha$, and finally, there are also leaves rising directly from the trunk. All these leaves constitute the finite contexts of $\tau$.
\end{minipage}

In this way, the discontinuity points of $g$ are the branches, $\underline 00$, $\underline 110$, $\underline 00010$,  $\underline 1100010$ and so on, as well as the trunk. In order to get the discontinuities, it is enough to alternate the value of $p$ on the leaves along the infinite branches. For example, for the first discontinuity point $0^{\infty}0$, take $p(1|10^{2k})=\epsilon=1-p(1|10^{2k+1}),k\geqslant1$, for some  $\epsilon\in(0,1/2)$. We can repeat the same method for all the leaves of the tree (all the finite contexts).

It is not difficult to see that this example fits into the conditions of Corollary \ref{prop:explicitVLMC}. First, we indeed have $\inf g=\varepsilon>0$. Second, if $n$ is fixed, $|\tau_n|$ is simply the number of infinite branches that started before $n$, plus one, for the trunk, that is, $|\tau_n|=\max\{k\ge1:|w(k_n)|< n\}+1$. A simple calculation then shows that $|\tau^n|\leqslant n$ and consequently $\bar{gr}({\tau})=1$ which is strictly smaller than $[1-(|A|-1)\varepsilon]^{-1}=(1-\varepsilon)^{-1}$ independently of the value of $\varepsilon$. We  conclude that this example has a unique compatible stationary measure which is weak-Bernoulli.

\subsection{An example of non-existence with $\inf g>0$}\label{sec:non-exist-positive}

In the literature, as far as we know, all the examples  of $g$-functions having no stationary compatible $g$-measure satisfy $\inf g=0$ (see for instance the example given in Subsection \ref{ex:1}, and other examples in \cite{cenacetal/2018}). Here we present an example, suggested to us by Noam Berger, of $g$-function satisfying $\inf g>0$ for which no stationary compatible measure exists. 

 First, let us introduce the sets
 \begin{align*}
 \mathcal X^-_>&:=\{\underline x:\limsup\frac{1}{n}\sum_{i=1}^nx_{-i}>1/2\}\\
 \mathcal X^-_{\leqslant}&:=\{\underline x:\limsup\frac{1}{n}\sum_{i=1}^nx_{-i}\leqslant1/2\}.
 \end{align*} 
 The $g$-function $\hat g$ is defined on  $A=\{0,1\}$ through
\begin{align*}
\hat g(\underline x1)&=0.3
\quad\hspace{0.3cm}\text{if} \quad\underline x\in\mathcal X^-_>\\
\hat g(\underline x1)&=0.7
\quad\hspace{0.3cm}\text{if}\quad \underline x\in\mathcal X^-_{\leqslant}.
\end{align*}
Clearly, $\hat g\geqslant0.3$. Moreover, it is everywhere discontinuous. Indeed, for any past $\underline x\in\mathcal X^-_>$, taking $\underline z\in\mathcal X^-_{\leqslant}$, we have $0.7=g(\underline zx_{-k}^{-1}1)\nrightarrow 0.3=g(\underline x1)$, and a parallel reasoning can be done for $\underline x\in \mathcal X^-_{\leqslant}$. 
\begin{Prop}
	There exists no stationary compatible measure for $\hat g$. 
\end{Prop}
%
\begin{proof}
 Suppose there exists a stationary compatible measure $\mu$. Since the set $\mathcal M_T(g)$  is convex there necessarily exists at least one extremal measure. Let us denote by $\hat\mu$ one such extremal measure which is therefore ergodic  \citep[Theorem 3.3]{fernandez/maillard/2005}. Since both, $X^-_>$ and $\mathcal X^-_{\leqslant}$ are shift invariant measurable sets (when seen as subsets of $\mathcal X$, that is, their cartesian product with $\XX^+$), it follows, by ergodicity, that either $\hat\mu(X^-_>)=0$ and $\hat\mu(\mathcal X^-_{\leqslant})=1$ or $\hat\mu(X^-_>)=1$ and $\hat\mu(\mathcal X^-_{\leqslant})=0$. On the other hand, by compatibility and using the definition of $\hat g$, we have
\begin{align*}
\hat\mu([a_1^n])&=\int_{\XX^-}\hat g_n(\underline xa_1^n)\hat\mu(d\underline x)\\
&=\hat\mu(X^-_>)\prod_{i=1}^np_{0.3}(a_i)+\hat\mu(\mathcal X^-_{\leqslant})\prod_{i=1}^np_{0.7}(a_i),
\end{align*}
in which we used the notation $p_{0.3}(1):=1-p_{0.3}(0)=0.3$ and $p_{0.7}(1):=1-p_{0.7}(0)=0.7$.
We proceed  by contradiction, and suppose $\hat\mu(\mathcal X^-_{\leqslant})=0$. In this case $\hat\mu([a_1^n])=\prod_{i=1}^np_{0.3}(a_i)$ for any $n\geqslant1$ and any $a_1^n\in A^n$, which means that $\hat\mu|_{\mathcal F^{[0,+\infty]}}$ (and therefore $\hat\mu$ itself, by invariance) is the product measure of Bernoulli $\mathcal B(0.3)$ distributed random variables. But in this case $\hat\mu(\mathcal X^-_{\leqslant})=1$, which is a contradiction. A similar contradiction arises if we initially assume $\hat\mu(\mathcal X^-_{\leqslant})=1$. Therefore, their cannot exists a stationary compatible measure. 
\end{proof}

 \section{Some further important remarks}\label{sec:discussion}

The present section intends to make a more complete picture of actual knowledge on the question of discontinuities in $g$-functions in the literature. 

\subsection{Discussion on the meaning of ``non-regular $g$-measure''}
As we mentioned earlier, the terminology ``non-regular $g$-measure'' was an abuse, meaning that we are working on measures compatible with non-regular $g$-functions. There is however a subtlety here: since the definition of compatibility of $\mu$ with $g$ is given \emph{modulo} a set of $\mu$-measure $0$, infinitely many $g$-functions specify $\mu$, many of them being discontinuous. 

The example of Subsection \ref{ex:1} is very enlightening on this subject. If we take $q_k$ vanishing monotonically to $0$, so slowly that $V(q)<\infty$, then we will have existence of a stationary measure $\mu$, independently of the value of $q_\infty$. So  if we put $q_\infty>0$, we have a discontinuity at $\underline 0$, but since this past has $\mu$-measure $0$, we can remove the discontinuity by putting $\bar q_\infty=0$ instead, and obtaining a new $g$-function, $\bar g$ say, with which $\mu$ is still compatible. The discontinuity at $\underline 0$ is called non-essential in the Gibbs/non-Gibbs literature \citep{enter/fernandez/sokal/1993}. A radical example is given in the next subsection, in which the $g$-function is non-essentially discontinuous everywhere. On the other hand, if we take $q_i$ equals $\epsilon_1\in(0,1)$ or $\epsilon_2\in(0,1)$ (with $\epsilon_1\ne\epsilon_2$) according to whether $i$ is odd or even, then there still exists a compatible measure $\mu$, but this time the discontinuity of $g$ at $\underline0$ cannot be removed by changing $g$ on a set of $\mu$-measure $0$. In this case the discontinuity is said to be essential. 

Coming back to our contribution here, our results concern essential as well as non-essential discontinuities: all we know is that there exists a $g$-measure (eventually unique and weak-Bernoulli) compatible with the given $g$-function.
%

\subsection{An everywhere non-essentially discontinuous example}\label{subsec:noness} 
As explained above, $g$-functions may be artificially discontinuous: a measure may be compatible with it in such a way that we can a posteriori transform $g$ to an everywhere continuous $g'$ by changing $g$ on a set of null $\mu$-measure. A particularly radical example is the following, on $A=\{0,1\}$, put
\begin{align*}
g(\underline x1)&=\epsilon
\quad\quad\hspace{0.3cm}\text{if $\sum_{i\geqslant1}\frac{x_{-i}}{2^i}\in\mathbb{Q}\cap[0,1[$}\\
g(\underline x1)&=1-\epsilon
\quad\text{if $\sum_{i\geqslant1}\frac{x_{-i}}{2^i}\in\mathbb{Q}^c\cap[0,1[$.}
\end{align*}
In words, the value of $g(\underline x1)$ depends on whether or not $\underline x$ codifies a rational number of $[0,1]$. 
It turns out that the product measure $\mu$ with time-independent marginal $\mu(1)=1-\epsilon$ is compatible with $g$. This follows  from the characterisation of \cite[Th\'eor\`eme 1 (iii)]{Ledrappier}: $\mu$ is compatible if its entropy $h(\mu)$ equals  $\int\log gd\mu$. In our case it is well-known that $h(\mu)=\epsilon\log\epsilon+(1-\epsilon)\log(1-\epsilon)$. On the other hand (we abuse notation by writing $\underline x\in\mathbb{Q}$ to mean that $\underline x$ codifies a rational number)
\begin{align*}
\int \log g(\underline xx_0)d\mu(\underline xx_0)&=\int\log g(\underline x1)\mu(1)d\mu(\underline x)+\int\log g(\underline x0)\mu(0)d\mu(\underline x)\\
&=(1-\epsilon)\left[ \mu(\underline x\in\mathbb{Q})\log \epsilon+\mu\left(\underline x\in[0,1[\setminus\mathbb{Q}\right)\log (1-\epsilon)\right]\\&\hspace{0.9cm}+\epsilon\left[ \mu(\underline x\in\mathbb{Q})\log (1-\epsilon)+\mu\left(\underline x\in[0,1[\setminus\mathbb{Q}\right)\log \epsilon\right]\\
&=\epsilon\log\epsilon+(1-\epsilon)\log(1-\epsilon),
\end{align*}
where we used independence in the first equality, the definition of $g$ in the second equality and finally the well-known fact that the set of infinite sequences that codify the rational numbers have  $\mu$-measure $0$ in the last equality. In fact, what we are saying is simply that, changing $g(\underline x1)$ to $1-\epsilon$ for any rational $\underline x$ yields $g\equiv1-\epsilon$ (which obviously has $\mu$ as compatible measure), and this change concerns a set of $\mu$-measure $0$.

\subsection{An everywhere essentially discontinuous $g$-measure}\label{sec:every} 

There  exist stationary measures which are compatible with {everywhere} essentially  discontinuous $g$-functions. A very simple example is known in the literature of non-Gibbs measures \citep[for instance]{lorinczi/1998} and seems to be  originally due \cite[Chapter 3.12]{furstenberg1960stationary}. Consider the alphabet $A=\{-1,+1\}$ and  take $\mu$ the product measure with marginal $\mu(-1)=\epsilon$ with $\epsilon\neq 1/2$. Then, consider the function $F:\mathcal{X}\rightarrow\mathcal{X}$  defined by $F(x)_i=x_{i-1}x_i$ (product of $x_{i-1}$ and $x_i$). The measure $\nu:=\mu\circ F^{-1}$ is stationary by construction, and can be proved to be everywhere  discontinuous. This is proven in \cite[p.38]{verbitskiy/15}. We prove here that it is everywhere essentially discontinuous. Recall first some facts about this example.

For $x\in\mathcal{X}^{\leqslant 0}$, define the density of $x$ (when it exists) as
\[
d(x)=\lim_{n\to\infty}\frac{\#\{-n\leqslant i\leqslant -1, x_i=-1\}}{n}.
\]
Because of the law of large numbers, the points $x$ generic for the measure $\mu$ are those for which $d(x)=\epsilon$.
If $y=\underline{y}y_0\in\mathcal{X}^{\leqslant 0}$ is fixed, $F^{-1}(y)$ contains two elements. Denote by $x^+(y)$ the element for which $x_0=+1$ and $x^-(y)$ the other one. Observe that the points $y$ generic for the measure $\nu$ are those for which $d(x^+)=\epsilon$ or $d(x^-)=\epsilon$; denote by $G$ this set. 
The $g$-function may be defined in the classical way
$$
g(y)=\lim_{n\to\infty}\nu(y_0|y_{-n}^{-1})
$$
when the limit exists (it exists for $\nu$-almost every $y$ by the reverse Martingale theorem).
\begin{Claim}
	For all $y\in G$ with $y_0=+1$, the $g$-function is well defined and equals
\begin{equation*}
	g(y)=\left\{\begin{array}{ccc}
		\epsilon&\text{if}&d(x^-(y))=\epsilon\\
		1-\epsilon&\text{if}&d(x^+(y))=\epsilon.
	\end{array}\right.
\end{equation*}
\end{Claim}
\begin{proof}
	Fix $y_{-n}^{0}$ with $y_0=+1$. Denote by $S_n$ the number of $+1$'s in $x^+(y)_{-1},\dots,x^+(y)_{-n-1}$, this is also the number of $+1$'s in $y_0,y_0y_{-1},\dots,y_0y_{-1}\dots y_{-n}$ (these are products, not strings). Then, a straightforward computation gives
	\[
	\nu(y_{-n}^0)=\epsilon^{1+S_n}(1-\epsilon)^{n+1-S_n}+\epsilon^{n+1-S_n}\epsilon^{1+S_n},
	\]
	and, as $y_0=+1$,
	\[
	\nu(y_{-n}^{-1})=\epsilon^{S_n}(1-\epsilon)^{n+1-S_n}+\epsilon^{n+1-S_n}\epsilon^{S_n}
	\]
	thus entailing, with $\rho=\frac{\epsilon}{1-\epsilon}$,
	\[
	\nu(y_0|y_{-n}^{-1})=(1-\epsilon)\frac{\rho+\rho^{2(n+1)\left(\frac12-\frac{S_n}{n+1}\right)}}{1+\rho^{2(n+1)\left(\frac12-\frac{S_n}{n+1}\right)}}.
	\]
	
	Suppose for example $\epsilon<\frac12$ or equivalently $\rho>1$. If $d(x^+(y))=\epsilon$, then the law of large numbers implies $\frac{S_n}{n+1}\to\epsilon$ and for $n$ large, $\frac12-\frac{S_n}{n+1}>0$, consequently $\nu(y_0|y_{-n}^{-1})\to 1-\epsilon$. If $d(x^-(y))=\epsilon$, then $\frac{S_n}{n+1}\to 1-\epsilon$ and for $n$ large, $\frac12-\frac{S_n}{n+1}<0$, consequently $\nu(y_0|y_{-n}^{-1})\to \epsilon$.
\end{proof}
On $\mathcal{X}^{\leqslant 0}\setminus G$, $g$ may be arbitrarily defined, as it a set of $\nu$ mesure zero.
\begin{Claim}
	The function $g$ is everywhere essentially discontinuous.
\end{Claim}
\begin{proof}
	Actually we prove that: $\forall y\in\mathcal{X}^{\leqslant 0}$, $\exists\delta>0, \forall n\in\NN$
	\[
	\exists V_1,V_2\subset [y_{-n}^0], \nu(V_1)>0, \nu(V_2)>0, \forall z_1\in V_1, \forall z_2\in V_2,\left|g(z_1)-g(z_2)\right|\geqslant\delta.
	\]
	This implies the discontinuity of $g$ at $y$. Moreover, if $f$ is such that $\nu(f\neq g)=0$, then the claim remains true for $f$, with $V_i\cap\{f=g\}$ and $y$ is also a discontinuity point of $f$. So the discontinuities are actually essential. 
	
	Consider $V_1=\{z\in G, z_{-n}^0=y_{-n}^0, d(x^+(z))=\epsilon\}$ and $V_2=\{z\in G, z_{-n}^0=y_{-n}^0, d(x^+(z))=1-\epsilon\}$; Firstly, $V_1\cup V_2=[y_{-n}^0]\cap G$. Secondly, $\nu(V_1)=\mu([(x^+(y))_{-n-1}^{0}])>0$ and $\nu(V_2)=\mu([(x^-(y))_{-n-1}^{0}])>0$. In addition, $g_{|V_1}=1-\epsilon$ and $g_{|V_2}=\epsilon$ giving the claim with $\delta=1-2\epsilon$.
\end{proof}

\subsection{Random context representation}

There exists an interesting characterisation, which was given in \cite{imbuzeiro/2015}, of compatible stationary measures that do not charge the set of discontinuities of $g$. Start with a stationary measure $\mu$ compatible with $g$. Then, the set of discontinuities of $g$ has $\mu$-measure $0$ if and only if the corresponding process admits a \emph{random context representation}. This means that there exist, for any $k\geqslant1$, measurable functions $\lambda_k: A^{\{-k,\ldots,-1\}}\rightarrow [0,1]$ and  $p_k: A\times A^{\{-k,\ldots,-1\}}\rightarrow [0,1]$
such that 
\begin{align*}
&\sum_k\lambda_k(x_{-k}^{-1})=1\,\,\,\mu{\text-a.e.}\,\,\underline x,\\
&\sum_{a\in A} p_k(a|x_{-k}^{-1})=1\,\,\,\forall x_{-k}^{-1},k\geqslant1\\
&g(\underline xa)=\sum_{k\geqslant1}\lambda_k(x_{-k}^{-1})p_k(a|x_{-k}^{-1})\,\,\,\mu{\text-a.e.}\,\,\underline x.\\
\end{align*}
This characterisation reminds us the  classical result of \cite{kalikow/1990} characterising  continuous $g$-measures in terms of convex mixture of $k$-steps Markov kernels, that is, the same as above, but holding for \emph{any} $\underline x$ instead of $\mu$-a.e. $\underline x$. 

On the one hand, we can combine this with our Theorem \ref{theo:exist} to obtain the following statement: for a discontinuous $g$-function $g$, if $\bar\mu^{\underline x}$ does not charge $\mathcal{D}_g$ for some $\underline x$, then $\bar\mu^{\underline x}$ is a stationary compatible measure, and it admits a random context representation. This may have implications for inference/estimation problems, and we refer to \cite{imbuzeiro/2015} on that part. On the other hand, it allows us to restate our Theorem \ref{theo:exist} as follows: ``if for some $\underline x$ we have that $\bar\mu^{\underline x}$ has a random context representation, then it is compatible with $g$.'' We are not sure whether this assumption is simpler to check however...

\subsection{Relation to (non-)Gibbsianity}

Without getting into details, let us only mention that the so-called \emph{Gibbs measures} of statistical mechanics are specified by strongly positive and continuous conditional probabilities with respect to both past and future \citep{kozlov/1974}. The function specifying these conditional probabilities is simply called \emph{specification}. It is  the two-sided counterpart of the $g$-function, since the later  specifies the conditional probabilities with respect to the past only. 
The question of non-Gibbsianity is the problem of finding the best framework (larger than that of Gibbs measures) to describe physical phenomena, allowing to include measures having discontinuous specifications. We refer to the encyclopedic review of \cite*{van1993regularity} for formal definitions and exhaustive description of the related issues. 

It is natural to ask whether both, continuous $g$-measures and Gibbs measure coincide. Are $g$-measures with continuous $g$-function always Gibbsian, and reciprocally, do Gibbs measures always have continuous $g$-function? Recently, this relation was completely clarified: neither continuity of the $g$-function implies Gibbsianity of the measure \citep{fernandez/gallo/maillard/2011} nor Gibbsianity of the measure implies it has a continuous $g$-function \citep{bissacot2018entropic}. The former used the renewal process with suitably chosen parameters as a counterexample. The later proved that the (Gibbs) measures specified by the model of \cite{dyson1969existence} at low temperature have discontinuous $g$-functions. 

There are still many open questions in this context, the interested reader is prompt to consult \cite{berghout2019relation} and the above references. 
 
\section{Proofs}\label{sec:proofs}

Many statements refer  to a  measure $\bar\mu^{\underline x}\in\mathcal M$ (see \eqref{eq:def_barmu}), defined through  a past $\underline x$ and a diverging subsequence of the natural numbers $m_k,k\geqslant1$. In the proofs, $\underline x$ and $m_k,k\geqslant1$ may be considered fixed once for all. 

\begin{proof}[Proof of Theorem \ref{theo:exist}] Assume the conditions of Theorem \ref{theo:exist}.
It is direct, by its definition as a Ces\`aro limit, that $\bar\mu^{\underline{x}}$ is stationary. Indeed, consider any measurable set $B\in\mathcal{F}$, then
\begin{align*}
\bar\mu^{\underline{x}} (T^{-1}B)&:=\lim_k\frac{1}{m_k}\sum_{i=0}^{m_k-1}\mu^{\underline{x},-i}(T^{-1}B)\\
&=\lim_k\frac{1}{m_k}\sum_{i=1}^{m_k}\mu^{\underline{x},-i}(B)\\
&=\lim_k\frac{1}{m_k}\left(\mu^{\underline{x},-m_k}(B)-\mu^{\underline{x},0}(B)+\sum_{i=0}^{m_k-1}\mu^{\underline{x},-i}(B)\right)\\
&=\bar\mu^{\underline{x}}(B).
\end{align*} 
So it remains to prove that it is compatible in the sense of \eqref{eq:compatible}. 

Since the measure $\bar\mu^{\underline{x}}$ is translation invariant, it is univocally defined by its restriction on $\mathcal{X}^{-}$, and we will make no difference in the notation between $\bar\mu^{\underline{x}}$ and its restriction on pasts. By the reverse Martingale theorem, there exists a set $S\subset \mathcal{X}^{-}$ of full $\bar\mu^{\underline{x}}$-measure for which 
\[
\EE_{\bar\mu^{\underline{x}}}(\1_{[a]_0}|\mathcal{F}^{[-l,-1]})(y)\stackrel{l\rightarrow\infty}{\rightarrow}\EE_{\bar\mu^{\underline{x}}}(\1_{[a]_0}|\mathcal{F}^{-})(y)
\]
for any $y$ such that $\underline y\in S$.
{We will consider this convergence on the smaller set $\mathcal{D}_g^c\cap S\cap\mathcal{R}$ where $\mathcal{R}\subset\mathcal{X}^-$ denotes the set of pasts $\underline y$ satisfying $\bar\mu^{\underline x}(y_{-n}^{-1})>0$ for any $n\geqslant1$. We can prove that $\bar\mu^{\underline x}(\mathcal{R})=1$. Indeed, if we let $m(\underline y):=\inf\{n\geqslant1:\bar\mu^{\underline x}(y_{-n}^{-1})=0\}$, we have 
\[
\bar\mu^{\underline x}(\mathcal{R}^c)=\sum_{i\geqslant1}\bar\mu^{\underline x}(\{\underline y:m(\underline y)=i\})=\sum_{i\geqslant1}\sum_{y_{-i}^{-1}:m(y_{-i}^{-1})=i}\bar\mu^{\underline x}(y_{-i}^{-1})=0.
\]
This means in particular that, under the assumption of the theorem, $\bar\mu^{\underline x}(\mathcal{D}_g^c\cap S\cap\mathcal{R})=1$.}
 
For any $y$ such that $\underline y\in\mathcal{D}_g^c\cap S\cap\mathcal{R}$ we can write $\EE_{\bar\mu^{\underline{x}}}(\1_{[a]_{0}}|\mathcal{F}^{[-l,-1]})(y)={\bar\mu^{\underline{x}}(y_{-l}^{-1}a)}/{\bar\mu^{\underline{x}}(y_{-l}^{-1})}$. 
So on this set we know that ${\bar\mu^{\underline{x}}(y_{-l}^{-1}a)}/{\bar\mu^{\underline{x}}(y_{-l}^{-1})}\stackrel{l\rightarrow\infty}{\rightarrow}\EE_{\bar\mu^{\underline{x}}}(\1_{[a]_{0}}|\mathcal{F}^{-})(y)$  and we will now prove that
\[
\frac{\bar\mu^{\underline{x}}(y_{-l}^{-1}a)}{\bar\mu^{\underline{x}}(y_{-l}^{-1})}\rightarrow g(\underline{y}a).
\] 
For $j\geqslant l$, 
\begin{align}\label{eq:mu1}
\mu^{\underline{x},-j}(y_{-l}^{-1}a)&=\sum_{z_{-j}^{-l-1}}\mu^{\underline{x},-j}(z_{-j}^{-l-1}y_{-l}^{-1})g(\underline xz_{-j}^{-l-1}y_{-l}^{-1}a).
\end{align}
Thus for any fixed $l\geqslant1$ and sufficiently large $k$'s
\begin{align}\label{eq:mu1c}
\bar\mu^{\underline{x},-m_k}(y_{-l}^{-1}a)&=\frac{1}{m_k}\left[\sum_{j=l}^{m_k}\sum_{z_{-j}^{-l-1}}\mu^{\underline{x},-j}(z_{-j}^{-l-1}y_{-l}^{-1})g(\underline xz_{-j}^{-l-1}y_{-l}^{-1}a)+o(m_k)\right].
\end{align}
Recall
\[
\text{var}^g_l(\underline{y}):=\sup_{z:z_{-l}^{-1}=y^{-1}_{-l}}|g(\underline za)-g(\underline{y}a)|,
\]
with this notation, we have  for any $z_{-j}^{-l-1}$
\begin{equation}\label{eq:mu2}
g(\underline{y}a)-\text{var}^g_l(\underline{y})\leqslant g(\underline xz_{-j}^{-l-1}y_{-l}^{-1}a)\leqslant g(\underline{y}a)+\text{var}^g_l(\underline{y}),
\end{equation}
and thus \eqref{eq:mu1c} can be bounded as follows
\begin{align}\label{eq:mu1c2}
\bar\mu^{\underline{x},-m_k}(y_{-l}^{-1}a)&\stackrel{\le}{\geqslant}\frac{1}{m_k}\left[\left(g(\underline{y}a)\stackrel{+}{-}\text{var}^g_l(\underline{y})\right)\sum_{j=l}^{m_k}\mu^{\underline{x},-j}(y_{-l}^{-1})+o(m_k)\right].
\end{align}
Taking now the limit in $k$ we obtain
\begin{equation*}
\bar\mu^{\underline{x}}(y_{-l}^{-1})\left(g(\underline{y}a)-\text{var}^g_l(\underline{y})\right)\le\bar\mu^{\underline{x}}(y_{-l}^{-1}a)\leqslant \bar\mu^{\underline{x}}(y_{-l}^{-1})\left(g(\underline{y}a)+\text{var}^g_l(\underline{y})\right).
\end{equation*}
Finally, recalling that $\underline{y}\in \mathcal{D}_g^c\cap S\cap\mathcal{R}$  we have that $\text{var}^g_l(\underline{y})\stackrel{l\rightarrow\infty}{\rightarrow}0$
\[
\frac{\bar\mu^{\underline{x}}(y_{-l}^{-1}a)}{\bar\mu^{\underline{x}}(y_{-l}^{-1})}\rightarrow g(\underline{y}a).
\]
Since $\mathcal{D}_g^c\cap S\cap\mathcal{R}$ has full $\bar\mu^{\underline{x}}$-measure, we have
\[
\EE_{\bar\mu^{\underline{x}}}(\1_{[a]_{0}}|\mathcal{F}^{-})(y)=g(\underline{y}a)\,,\,\,\,\text{$\bar\mu^{\underline{x}}$-a.s}
\]
concluding the proof.
\end{proof}

Before we come to the proof of  Corollary \ref{cor1}, let us state and prove a simple lemma.
\begin{Lem}\label{lemma:proof_claim}
Suppose $g\geqslant\epsilon$. Then for any $k\in\mathbb Z,n\geqslant1$ and any $x_{k+1}^{k+n}\in A^n$
\begin{equation}\label{eq:claim}
\epsilon^n\le\bar\mu^{\underline{x}}(x_{k+1}^{k+n})\le(1-\epsilon)^n.
\end{equation}
\end{Lem}

\begin{proof}
For any $j\geqslant -k\vee0$
\begin{align*}
\mu^{\underline{x},-j}(x_{k+1}^{k+n})&=\sum_{z_{-j}^{k}}\mu^{\underline{x},-j}(z_{-j}^{k})g_n(\underline xz_{-j}^{k}x_{k+1}^{k+n})\geqslant \epsilon^n\sum_{z_{-j}^{k-1}}\mu^{\underline{x},-j}(z_{-j}^{k})\geqslant\epsilon^n
\end{align*}
which implies also that $\bar\mu^{\underline{x},-m_k}(x_{k+1}^{k+n})\geqslant\epsilon^n$ since it is a Ces\`aro mean of $\mu^{\underline{x},-j}(x_{k+1}^{k+n})$'s. Since $g\geqslant \epsilon$ we also necessarily have $g\leqslant 1-\epsilon$, and applying the same reasoning and reversing the inequality, we conclude the proof. 
\end{proof}

 \begin{proof}[Proof of Corollary \ref{cor1}]
It is enough to prove that $\bar\mu^{\underline{x}}$ is non-atomic, which is direct using Lemma \ref{lemma:proof_claim}: $\bar\mu^{\underline{x}}(\underline{y})=\lim_n\bar\mu^{\underline{x}}(y_{-n}^{-1})\le\lim_n(1-\epsilon)^n=0$.
 \end{proof}

\begin{proof}[Proof of Corollary \ref{prop1}] 
We have, for any $\underline x$ and  $j\geqslant n$
\begin{align*}
\sum_{y_{-n}^{-1}\in \mathcal{D}_g^n}\mu^{\underline{x},-j}(y_{-n}^{-1})&=\sum_{y_{-n}^{-1}\in \mathcal{D}_g^n}\sum_{z_{-j}^{-n-1}}\mu^{\underline{x},-j}(z_{-j}^{-n-1})g_n(\underline xz_{-j}^{-n-1}y_{-n}^{-1})\\
&=\sum_{z_{-j}^{-n-1}}\mu^{\underline{x},-j}(z_{-j}^{-n-1})\sum_{y_{-n}^{-1}\in \mathcal{D}_g^n}g_n(\underline xz_{-j}^{-n-1}y_{-n}^{-1}).
\end{align*}
According to the definition of $P_g(\mathcal{D}_g)$, for any $\delta$,  we have for sufficiently large $n$
\begin{align*}
\sum_{y_{-n}^{-1}\in \mathcal{D}_{g}^n}g_n(\underline xz_{-j}^{-n-1}y_{-n}^{-1})\le
\sum_{y_{-n}^{-1}\in \mathcal{D}_{g}^n}\sup_{\underline x}g_n(\underline xz_{-j}^{-n-1}y_{-n}^{-1})\le\left(e^{P_g(\mathcal{D}_g)+\delta}\right)^n.
\end{align*}
Therefore we have for sufficiently large $j$'s
\[
\sum_{y_{-n}^{-1}\in \mathcal{D}_g^{n}}\mu^{\underline{x},-j}(y_{-n}^{-1})\leqslant \left(e^{P_g(\mathcal{D}_g)+\delta}\right)^n.
\]
But then, the same holds under $\bar\mu^{\underline{x}}$ since
\begin{align*}
\sum_{y_{-n}^{-1}\in  \mathcal{D}_{g}^n}\bar\mu^{\underline{x}}(y_{-n}^{-1})&=\sum_{y_{-n}^{-1}\in  \mathcal{D}_{g}^n}\lim_k\frac{1}{m_k}\sum_{j=0}^{m_k-1}\mu^{\underline{x},-j}(y_{-n}^{-1})\\
&=\lim_k\frac{1}{m_k}\sum_{j=0}^{m_k-1}\sum_{y_{-n}^{-1}\in  \mathcal{D}_{g}^n}\mu^{\underline{x},-j}(y_{-n}^{-1})\\
&\le\lim_k\frac{1}{m_k}\left[n+(m_k-n)\left(e^{P_g(\mathcal{D}_g)+\delta}\right)^n\right]\\
&=\left(e^{P_g(\mathcal{D}_g)+\delta}\right)^n.
\end{align*}
Under the assumption of the corollary we know that $P_g(\mathcal{D}_g)<0$,  therefore, choosing $\delta<-P_g(\mathcal{D}_g)$ we conclude
\begin{align*}
\bar\mu^{\underline{x}}(\mathcal{D}_g)=\bar\mu^{\underline{x}}\left(\bigcap_{n\geqslant1}\{y_{-n}^{-1}\in  \mathcal{D}_{g}^n\}\right)=\lim_n\sum_{y_{-n}^{-1}\in  \mathcal{D}_{g}^n}\bar\mu^{\underline{x}}(y_{-n}^{-1})=\lim_n\left(e^{P_g(\mathcal{D}_g)+\delta}\right)^n=0.
\end{align*}
\end{proof}

\begin{proof}[Proof of Corollary \ref{cor2}]
We  prove the result in the case where the string $v$ has size $1$, that is $v\in A$ is a symbol. The extension to the case of string of more symbols is straightforward. Observe that for any $n\geqslant1$, $\mathcal{D}_g^n=(A\setminus\{v\})^n$ is the set of strings of size $n$ of symbols from $A\setminus\{v\}$. Recall also that we assume that there exists some $\epsilon>0$ such that $\inf_{\underline x}g(\underline xv)=\epsilon$. 

 Proceeding exactly as in the proof of Corollary \ref{prop1}, it is enough to show that for some $\underline x$, $\sum_{y_{-n}^{-1}\in \mathcal{D}_{g}^n}g_n(\underline xz_{-j}^{-n-1}y_{-n}^{-1})$ vanishes as $n$ diverges uniformly in $z_{-j}^{-n-1}$. This is the case as the following computation shows
\begin{align*}
\sum_{y_{-n}^{-1}\in \mathcal{D}_g^n}g_n(\underline xz_{-j}^{-n-1}y_{-n}^{-1})&=
\sum_{y_{-n}^{-2}\in \mathcal{D}_g^{n-1}}g_{n-1}(xz_{-j}^{-n-1}y_{-n}^{-2})\sum_{y_{-1}\in A\setminus\{v\}}g(\underline xz_{-j}^{-n-1}y_{-n}^{-1})\\
&=\sum_{y_{-n}^{-2}\in \mathcal{D}_g^{n-1}}g_{n-1}(\underline xz_{-j}^{-n-1}y_{-n}^{-2})[1-g(\underline xz_{-j}^{-n-1}y_{-n}^{-2}v)]\\
&\leqslant (1-\epsilon)\sum_{y_{-n}^{-2}\in \mathcal{D}_g^{n-1}}g_{n-1}(\underline xz_{-j}^{-n-1}y_{-n}^{-2}),
\end{align*}
which we apply $n$ times to get
\begin{align*}
\sum_{y_{-n}^{-1}\in \mathcal{D}_g^{n}}g_n(\underline xz_{-j}^{-n-1}y_{-n}^{-1})\leqslant (1-\epsilon)^n.
\end{align*}

 \end{proof}

\begin{proof}[Proof of Theorem \ref{exist_uniq_VLMC}] The proof of the existence result is simple. First observe that $\bar\mu^{\underline{x}}(\mathcal{D}_g)\leqslant \bar\mu^{\underline{x}}(\ell=\infty)$ since $\mathcal{D}_g\subset\{\underline x:\ell(\underline x)=\infty\}$. Under our assumption, we therefore have $\bar\mu^{\underline{x}}(\ell=\infty)=0$, which by Theorem \ref{theo:exist} implies existence. 

For the proof of the uniqueness result,  observe that for a context tree with context length function $\ell$, we have for any past $\underline y$
\[
\text{var}^g_n(\underline y)\leqslant \1_{(n,\infty)}(\ell(\underline y)).
\]
\cite{johansson/oberg/2003} proved that, if $\inf g>0$ and if for some $g$-measure $\mu$ we have 
\[
\EE_\mu(\sum_n(\text{var}^g_n)^2)<\infty
\]
then it is the unique $g$-measure. Then, the proof of the proposition follows using the monotone convergence theorem:
\begin{align*}
\EE_{\bar\mu^{\underline{x}}}(\sum_n(\text{var}^g_n)^2)&=\sum_{n\geqslant0}\EE_{\bar\mu^{\underline{x}}}(\text{var}^g_n)^2 \le\sum_{n\geqslant0}\EE_{\bar\mu^{\underline{x}}}\1_{(n,\infty)}(\ell)=\sum_{n\geqslant0}\bar\mu^{\underline{x}}(\ell>n)=\EE_{\bar\mu^{\underline{x}}}\ell<\infty.
\end{align*}
\end{proof}

\begin{proof}[Proof of Theorem  \ref{theo:unifo}]
Observe that
\begin{align*}
\bar\mu^{\underline{x}}(\ell=\infty)&=\lim_i\bar\mu^{\underline{x}}(\ell>i)\\
&=\lim_i\lim_k\frac{1}{m_k}\sum_{j=0}^{m_k-1}\mu^{\underline{x},-j}(\{z_0^{\infty}:\ell(z_0^{i-1})>i\})\\
&=\lim_i\lim_k\frac{1}{m_k}\sum_{j=0}^{m_k-1}\sum_{y_{-j}^{-1}}\mu^{\underline{x},-j}(y_{-j}^{-1})\mu^{\underline{x}y_{-j}^{-1}}(\{z_0^{\infty}:\ell(z_0^{i-1})>i\})\\
&\leqslant \lim_i\lim_k\frac{1}{m_k}\sum_{j=0}^{m_k-1}\sup_{\underline x}\mu^{\underline{x}}(\{z_0^{\infty}:\ell(z_0^{i-1})>i\}).
\end{align*}
This converges to $0$ under our assumption, implying that $\bar\mu^{\underline{x}}(\ell=\infty)=0$ and, by Theorem \ref{exist_uniq_VLMC}, that existence is granted. 

The proofs of uniqueness and   weak-Bernoullicity are very similar to the proofs given by \cite{gallesco/gallo/takahashi/2018}. However, due to substancial differences in the assumptions (essentially because we do not assume $\inf g>0$ but only $g>0$) we thought it was easier and clearer for the reader if we give the complete proof here. 

Let   $\mathcal F^+:=\mathcal F^{[0,+\infty]}$. The proofs of uniqueness and weak-Bernoullicity will be done  proving that, (a) for any two stationary measures $\mu$ and $\nu$ compatible with $g$, we have $\mu|_{\mathcal F^+}\ll\nu|_{\mathcal F^+}$ and (b) for any pair of pasts $\underline x,\underline y\in\XX^-$ we have $\mu^{\underline x}|_{\mathcal F^+}\ll\mu^{\underline y}|_{\mathcal F^+}$. 
We then use Lemma \ref{lem:uniq} (stated and proved below) to conclude uniqueness from (a), and we invoke the implication (iii)$\Rightarrow$(i) of \cite[Corollary 2.8]{tong/handel/2014} to conclude weak-Bernoullicity from (b). 

So it only remains to prove the absolute continuity affirmations (a) and (b). We do this using a very general theorem of \cite{jacod/shiryaev/2002} that we state after some further definitions. First, for any two measures $\mu,\nu\in\mathcal M$, and  $z \in \XX$, let $Z_n(z) := \frac{d\mu|_{\F^{[0,n]}}}{d\nu|_{\F^{[0,n]}}}(z)$ and $\alpha_n(z) := Z_n(z)/Z_{n-1}(z)$. Finally, let
\begin{equation*}
 d_n(z) := E_{\nu}\big[(1-\sqrt{\alpha_n})^2 | \F^{[0, n-1]}\big](z). 
\end{equation*}
Our tool to prove absolute continuity is the following theorem.
\begin{Theo}[see \cite{jacod/shiryaev/2002}, Theorem 2.36, p.253]\label{theo:jacob}
If for all $n \geqslant 0$ we have $\mu |_{\F^{[0,n]}} \ll \nu |_{\F^{[0,n]}}$, then $\mu|_{\F^+} \ll \nu|_{\F^+}$ if and only if $\sum_{n=1}^\infty d_n < \infty$, $\mu$-a.s.
\end{Theo}
According to what we said before, for our proofs,  $\mu$ and $\nu$ are either two stationary measures compatible with $g$ or the measures $\mu^{\underline x}$ and $\mu^{\underline y}$ constructed using $g$ through \eqref{eq:ionescu}. In any case,  since we assume $g>0$, we have $\mu(z_0^n)>0$ and $\nu(z_0^n)>0$ for any $n\geqslant0$ and any $z_0^n\in A^{n+1}$, so that the assumption that  $\mu |_{\F^{[0,n]}} \ll \nu |_{\F^{[0,n]}},n\geqslant0$ actually holds in both situations. On the other hand, since $g$ is a probabilistic context tree, if $\ell^g(z_0^{n-1})\leqslant n$, we have $\mu(z_n|z_0^{n-1})=\nu(z_n|z_0^{n-1})=p^g(z_n|c_{\tau}(z_0^{n-1}))$ (see \eqref{eq:compatibleVLMC}). Thus
\begin{align*}
\alpha_n(z)&:=\frac{\mu(z_n|z_0^{n-1})}{\nu(z_n|z_0^{n-1})}\\&=\1_{[1,n]}(\ell^g(z_0^{n-1}))+\1_{(n,\infty)}(\ell^g(z_0^{n-1}))\frac{\mu(z_n|z_0^{n-1})}{\nu(z_n|z_0^{n-1})}\\&\geqslant\1_{[1,n]}(\ell^g(z_0^{n-1})).
\end{align*}
We get $(1-\sqrt{\alpha_n(z)})^2\leqslant \1_{(n,\infty)}(\ell^g(z_0^{n-1}))$.  Thus (we use  the notation $\ell^{g,i}(z):=\ell^g(z_{-\infty}^i)$ for any $i\in\mathbb Z$)
\begin{align*}
d_n(z)&:=\mathbb E_{\nu}((1-\sqrt{\alpha_n})^2|\F^{[0, n-1]})(z)\\
&\leqslant \mathbb E_{\nu}(\1_{(n,\infty)}(\ell^{g,n-1})|\F^{[0, n-1]})(z)\\
&=\1_{(n,\infty)}(\ell^{g,n-1}(z))\\
&=\1_{(n,\infty)}(\ell^g(z_0^{n-1}))
\end{align*}
where, in the penultimate line we used the fact that the event $\{\ell^g_{n-1}> n\}$ is $\F^{[0, n-1]}$-measurable.
When  $\mu=\mu^{\underline x}$ and $\nu=\mu^{\underline y}$, using the statement of Theorem \ref{theo:jacob} together with the assumption \eqref{eq:assumption-iii-us}, we conclude that $\mu^{\underline x}|_{\F^+} \ll \mu^{{\underline y}}|_{\F^+}$, proving the weak-Bernoullicity. The same reasoning would prove $\mu|_{\F^+} \ll \mu|_{\F^+}$ if \eqref{eq:assumption-iii-us} holds under any stationary  measure $\mu$ compatible with $g$, instead of $\mu^{\underline x}$. But since \eqref{eq:assumption-iii-us}  holds for any $\underline x$, it actually holds under $\mu$ as is seen integrating over all possible past $\underline x$.

\end{proof}
\begin{Lem}\label{lem:uniq}
Suppose that, for any two stationary measures $\mu$ and $\nu$ compatible with $g$ we have $\mu|_{\mathcal F^+}\ll\nu|_{\mathcal F^+}$. Then there exists actually only one stationary compatible measure.
\end{Lem}
\begin{proof}
Fix $\mu$ and $\nu$ stationary compatible measures. First, observe that the tail sigma field $\mathcal T^+:=\cap_{n\geqslant0}F^{[n,+\infty]}$ belongs to $\mathcal F^+$. Thus, $\mu|_{\mathcal T^+}\ll \nu|_{\mathcal T^+}$ as well as $\nu|_{\mathcal T^+}\ll \mu|_{\mathcal T^+}$.   Since these measures are weak-Bernoulli,  they are trivial on $\mathcal T^+$ \citep{bradley/2005}, which in turn implies that $\mu|_{\mathcal T^+}= \nu|_{\mathcal T^+}$. Invoking now (c)$\Rightarrow$(b) in Theorem 9.4 of Chapter 4 in \cite{thorisson/2000}, $\mu|_{\mathcal T^+}= \nu|_{\mathcal T^+}$ implies $|\mu(T^nC)-\nu(T^nC)|\rightarrow0$ for any cylinder set $C\in\mathcal{F}^{+}$. But since $\mu$ and $\nu$ are stationary, this means that $\mu(C)=\nu(C)$, meaning that $\mu|_{\mathcal{F}^{+}}=\nu|_{\mathcal{F}^{+}}$. Once more, by stationarity, we conclude that  $\mu=\nu$. 
\end{proof}
\begin{proof}[Proof of Corollary \ref{prop:explicitVLMC}]
Consider a $g$-function $g=(\tau,p)$ and denote by $\ell$ the related length function. Observe that $\mathcal{D}_g^n\subset \tau^n$. It follows therefore that $\bar{gr}({\mathcal{D}})\le\bar{gr}({\tau})$, and owing to the discussion preceding Section \ref{sec:preceding}, we conclude that existence holds.  In order to prove the remaining statements, we will prove that the condition of the proposition implies that of the second statement of Theorem \ref{theo:unifo}. It is enough to prove that 
\[
\bar{gr}({\tau})<[1-(|A|-1)\varepsilon]^{-1}\Rightarrow\sum_{i\geqslant1}\sup_{\underline x}\mu^{\underline{x}}(\{y_0^{\infty}:\ell(y_{0}^{i-1})>i\})<\infty
\]
since the latter, by Borel-Cantelli, would imply that we are in force of the assumption of Theorem \ref{theo:unifo}.

Observe that, for  any $\alpha\in(0,1)$ there exists $N_\alpha$ s.t. for $i>N_\alpha$ we have $|\tau^i|\leqslant (1-(|A|-1)\varepsilon)^{-i(1-\alpha)}$. Thus
\begin{align*}
\mu^{\underline{x}}(\{y_0^{\infty}:\ell(y_{0}^{i-1})>i\})&=\sum_{y_0^{i-1}:\ell(y_0^{i-1})>i}\mu^{\underline{x}}(y_0^{i-1})\\&\le|\tau^i|(1-(|A|-1)\varepsilon)^i\\
&\le(1-(|A|-1)\varepsilon)^{-i(1-\alpha)}(1-(|A|-1)\varepsilon)^i\rightarrow0.\\
\end{align*}
Observe that this last quantity is independent of $x$ therefore
\[
\sum_{i\geqslant1}\sup_{\underline x}\mu^{\underline{x}}(\{y_0^{\infty}:\ell(y_{0}^{i-1})>i\})\leqslant N_\alpha+\sum_{i>N_\alpha}(1-(|A|-1)\varepsilon)^{i\alpha}<\infty.
\]

\end{proof}
\appendix

\section{Countability of  subtrees}

In this appendix we include the proofs of two {statements} concerning {the countability of subsets of $\mathcal X^-$, which we mentioned  in Subsection \ref{sec:preceding}.} In these two lemmas, we will use the generic name $\mathcal D$ for such subsets. By \emph{growth function} of a given $\mathcal D\subset\mathcal X^-$, we mean the sequence of natural numbers 
\[
d(n):=|\mathcal D^n|=|\{x_{-n}^{-1}:\underline x\in\mathcal{D}\}|,\,n\ge1.
\]
\begin{Lem}\label{lemma:tree1}
	For all $f:\NN\to\NN^*$ such that $\lim_{n\to\infty}f(n)=+\infty$, there exists an uncountable set $\mathcal D$ with growth function {$d(n)=f(n),n\ge1$.}
\end{Lem}

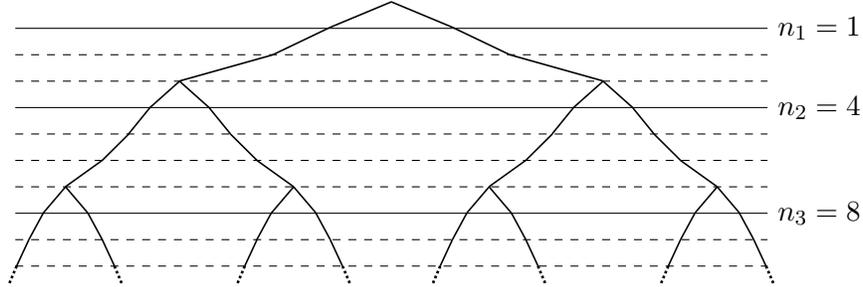
\begin{figure}[ht]
\begin{center}
\begin{tikzpicture}[scale=0.5]
\tikzset{edge from parent/.append style={very thick}}
\tikzset{every leaf node/.style={draw=none,circle=none},every internal node/.style={draw,circle,fill,scale=0.01}}
\tikzset{level distance=20pt}
\Tree [.{} 
	[.{} 
		[.{} 
			[.{} 
				[.{} 
					[.{} 
						[.{} 
							[.{} 
								[.{} 
									[.{} [.{} \edge[line width=2pt,dotted];\node[fill=none,draw=none]{};\edge[line width=2pt,draw=none];\node[fill=none,draw=none]{}; ] \edge[line width=2pt,draw=none];\node[fill=none,draw=none]{}; ] 
									\edge[line width=2pt,draw=none];\node[fill=none,draw=none]{}; 
								] 
								[.{} 
									\edge[line width=2pt,draw=none];\node[fill=none,draw=none]{}; 
									[.{} \edge[line width=2pt,draw=none];\node[fill=none,draw=none]{}; [.{} \edge[line width=2pt,draw=none];\node[fill=none,draw=none]{};\edge[line width=2pt,dotted];\node[fill=none,draw=none]{}; ] ] 
								] 
							] 
							\edge[line width=2pt,draw=none];\node[fill=none,draw=none]{}; 
						] 
						\edge[line width=2pt,draw=none];\node[fill=none,draw=none]{}; 
					] 
					\edge[line width=2pt,draw=none];\node[fill=none,draw=none]{}; 
				] 
				[.{} 
					\edge[line width=2pt,draw=none];\node[fill=none,draw=none]{}; 
					[.{} 
						\edge[line width=2pt,draw=none];\node[fill=none,draw=none]{}; 
						[.{} 
							\edge[line width=2pt,draw=none];\node[fill=none,draw=none]{}; 
							[.{} 
								[.{} 
									[.{} [.{} \edge[line width=2pt,dotted];\node[fill=none,draw=none]{};\edge[line width=2pt,draw=none];\node[fill=none,draw=none]{}; ] \edge[line width=2pt,draw=none];\node[fill=none,draw=none]{}; ] 
									\edge[line width=2pt,draw=none];\node[fill=none,draw=none]{}; 
								] 
								[.{} 
									\edge[line width=2pt,draw=none];\node[fill=none,draw=none]{}; 
									[.{} \edge[line width=2pt,draw=none];\node[fill=none,draw=none]{}; [.{} \edge[line width=2pt,draw=none];\node[fill=none,draw=none]{};\edge[line width=2pt,dotted];\node[fill=none,draw=none]{}; ] ] 
								] 
							] 
						] 
					] 
				] 
			] 
			\edge[line width=2pt,draw=none];\node[fill=none,draw=none]{}; 
		] 
		\edge[line width=2pt,draw=none];\node[fill=none,draw=none]{}; 
	] 
	[.{} 
		\edge[line width=2pt,draw=none];\node[fill=none,draw=none]{}; 
		[.{} 
			\edge[line width=2pt,draw=none];\node[fill=none,draw=none]{}; 
			[.{} 
				[.{} 
					[.{} 
						[.{} 
							[.{} 
								[.{} 
									[.{} [.{} \edge[line width=2pt,dotted];\node[fill=none,draw=none]{};\edge[line width=2pt,draw=none];\node[fill=none,draw=none]{}; ] \edge[line width=2pt,draw=none];\node[fill=none,draw=none]{}; ] 
									\edge[line width=2pt,draw=none];\node[fill=none,draw=none]{}; 
								] 
								[.{} 
									\edge[line width=2pt,draw=none];\node[fill=none,draw=none]{}; 
									[.{} \edge[line width=2pt,draw=none];\node[fill=none,draw=none]{}; [.{} \edge[line width=2pt,draw=none];\node[fill=none,draw=none]{};\edge[line width=2pt,dotted];\node[fill=none,draw=none]{}; ] ] 
								] 
							] 
							\edge[line width=2pt,draw=none];\node[fill=none,draw=none]{}; 
						] 
						\edge[line width=2pt,draw=none];\node[fill=none,draw=none]{}; 
					] 
					\edge[line width=2pt,draw=none];\node[fill=none,draw=none]{}; 
				] 
				[.{} 
					\edge[line width=2pt,draw=none];\node[fill=none,draw=none]{}; 
					[.{} 
						\edge[line width=2pt,draw=none];\node[fill=none,draw=none]{}; 
						[.{} 
							\edge[line width=2pt,draw=none];\node[fill=none,draw=none]{}; 
							[.{} 
								[.{} 
									[.{} [.{} \edge[line width=2pt,dotted];\node[fill=none,draw=none]{};\edge[line width=2pt,draw=none];\node[fill=none,draw=none]{}; ] \edge[line width=2pt,draw=none];\node[fill=none,draw=none]{}; ] 
									\edge[line width=2pt,draw=none];\node[fill=none,draw=none]{}; 
								] 
								[.{} 
									\edge[line width=2pt,draw=none];\node[fill=none,draw=none]{}; 
									[.{} \edge[line width=2pt,draw=none];\node[fill=none,draw=none]{}; [.{} \edge[line width=2pt,draw=none];\node[fill=none,draw=none]{};\edge[line width=2pt,dotted];\node[fill=none,draw=none]{}; ] ] 
								] 
							] 
						] 
					] 
				] 
			] 
		] 
	] 
      ]
\draw (-10,-20pt)--++(20,0) node[right]{$n_1=1$};
\draw[dashed] (-10,-40pt)--++(20,0);
\draw[dashed] (-10,-60pt)--++(20,0);
\draw (-10,-80pt)--++(20,0) node[right]{$n_2=4$};
\draw[dashed] (-10,-100pt)--++(20,0);
\draw[dashed] (-10,-120pt)--++(20,0);
\draw[dashed] (-10,-140pt)--++(20,0);
\draw (-10,-160pt)--++(20,0) node[right]{$n_3=8$};
\draw[dashed] (-10,-180pt)--++(20,0);
\draw[dashed] (-10,-200pt)--++(20,0);
\end{tikzpicture}
\end{center}
\caption{Illustration of the proof of Lemma~\ref{lemma:tree1}}.\label{fig:tree1}
\end{figure}

\begin{proof}
	{Fix a diverging sequence $f(n),n\ge1$}. Let $n_k$ be a strictly increasing sequence  {defined as follows: for any $k\ge1$, $n_k$ is such that $f(n)\geqslant 2^k$ if $n\geqslant n_k$. Now take a set
		\[
		\mathcal{D}:=\{B_{-\infty}^{-1}:B_{-i}\in\{0^{n_{i+1}-n_i},1^{n_{i+1}-n_i}\},i\geqslant1\}
		\]
		in which $B_{-\infty}^{-1}$ describes the concatenation of infinitely many finite binary strings $B_{-i}$'s.  This set is illustrated in Figure~\ref{fig:tree1}. By definition, it is in bijection with $\{0,1\}^{-\mathbb N^\star}$, thus uncountable}.  Moreover, it has the property that $d(n)=2^k$ for $n_k\leqslant n<n_{k+1}$, thus $d(n)\leqslant f(n)$. So naturally, we can modify the set by adding elements to $\mathcal D$ so that $d(n)= f(n)$  keeping uncountability. 
\end{proof}

\begin{figure}[ht]
\begin{center}
\begin{tikzpicture}[scale=0.25]
\tikzset{edge from parent/.append style={very thick}}
\tikzset{every leaf node/.style={draw=none,circle=none},every internal node/.style={draw,circle,fill,scale=0.01}}
\tikzset{level distance=40pt}
\Tree [.{} 
			[.{} 
				[.{} 
					[.{} [.{} [.{} [.{} [.{} \edge[line width=2pt,dotted];\node[fill=none,draw=none]{}; \edge[line width=2pt,draw=none];\node[fill=none,draw=none]{}; ] \edge[line width=2pt,draw=none];\node[fill=none,draw=none]{}; ] \edge[line width=2pt,draw=none];\node[fill=none,draw=none]{}; ] \edge[line width=2pt,draw=none];\node[fill=none,draw=none]{}; ] \edge[line width=2pt,draw=none];\node[fill=none,draw=none]{};
					] 
					\edge[line width=2pt,draw=none];\node[fill=none,draw=none]{}; 
				] 
			[.{} [.{} [.{} [.{} [.{} [.{} \edge[line width=2pt,dotted];\node[fill=none,draw=none]{}; \edge[line width=2pt,draw=none];\node[fill=none,draw=none]{}; ] \edge[line width=2pt,draw=none];\node[fill=none,draw=none]{}; ] \edge[line width=2pt,draw=none];\node[fill=none,draw=none]{}; ] \edge[line width=2pt,draw=none];\node[fill=none,draw=none]{}; ] \edge[line width=2pt,draw=none];\node[fill=none,draw=none]{}; ] \edge[line width=2pt,draw=none];\node[fill=none,draw=none]{};
			] 
	  		] 
	[.{} 
		[.{} 
			[.{} 
				[.{} [.{} [.{} [.{} \edge[line width=2pt,dotted];\node[fill=none,draw=none]{}; \edge[line width=2pt,draw=none];\node[fill=none,draw=none]{}; ] \edge[line width=2pt,draw=none];\node[fill=none,draw=none]{}; ] \edge[line width=2pt,draw=none];\node[fill=none,draw=none]{}; ] \edge[line width=2pt,draw=none];\node[fill=none,draw=none]{}; ] 
				[.{} [.{} [.{} [.{} \edge[line width=2pt,dotted];\node[fill=none,draw=none]{}; \edge[line width=2pt,draw=none];\node[fill=none,draw=none]{}; ] \edge[line width=2pt,draw=none];\node[fill=none,draw=none]{}; ] \edge[line width=2pt,draw=none];\node[fill=none,draw=none]{}; ] \edge[line width=2pt,draw=none];\node[fill=none,draw=none]{}; ] 
			]
		  	[.{} 
		  		[.{} [.{} [.{} [.{} \edge[line width=2pt,dotted];\node[fill=none,draw=none]{}; \edge[line width=2pt,draw=none];\node[fill=none,draw=none]{}; ] \edge[line width=2pt,draw=none];\node[fill=none,draw=none]{}; ] \edge[line width=2pt,draw=none];\node[fill=none,draw=none]{}; ] \edge[line width=2pt,draw=none];\node[fill=none,draw=none]{}; ]
		  		[.{} [.{} [.{} [.{} \edge[line width=2pt,dotted];\node[fill=none,draw=none]{}; \edge[line width=2pt,draw=none];\node[fill=none,draw=none]{}; ] \edge[line width=2pt,draw=none];\node[fill=none,draw=none]{}; ] \edge[line width=2pt,draw=none];\node[fill=none,draw=none]{}; ] \edge[line width=2pt,draw=none];\node[fill=none,draw=none]{}; ] 
		  	]
	    ]
	    [.{} [.{} 
			[.{} [.{} 
					[.{} [.{} \edge[line width=2pt,dotted];\node[fill=none,draw=none]{}; \edge[line width=2pt,draw=none];\node[fill=none,draw=none]{}; ] [.{} \edge[line width=2pt,dotted];\node[fill=none,draw=none]{}; \edge[line width=2pt,draw=none];\node[fill=none,draw=none]{}; ] ] 
					[.{} [.{} \edge[line width=2pt,dotted];\node[fill=none,draw=none]{}; \edge[line width=2pt,draw=none];\node[fill=none,draw=none]{}; ] [.{} \edge[line width=2pt,dotted];\node[fill=none,draw=none]{}; \edge[line width=2pt,draw=none];\node[fill=none,draw=none]{}; ] ] 
				 ]
			     [.{} 
			     	[.{} [.{} \edge[line width=2pt,dotted];\node[fill=none,draw=none]{}; \edge[line width=2pt,draw=none];\node[fill=none,draw=none]{}; ] [.{} \edge[line width=2pt,dotted];\node[fill=none,draw=none]{}; \edge[line width=2pt,draw=none];\node[fill=none,draw=none]{}; ] ] 
			     	[.{} [.{} \edge[line width=2pt,dotted];\node[fill=none,draw=none]{}; \edge[line width=2pt,draw=none];\node[fill=none,draw=none]{}; ] [.{} \edge[line width=2pt,dotted];\node[fill=none,draw=none]{}; \edge[line width=2pt,draw=none];\node[fill=none,draw=none]{}; ] ] 
			     ] 
			] 
			[.{} 
				[.{} 
					[.{} [.{} \edge[line width=2pt,dotted];\node[fill=none,draw=none]{}; \edge[line width=2pt,draw=none];\node[fill=none,draw=none]{}; ] [.{} \edge[line width=2pt,dotted];\node[fill=none,draw=none]{}; \edge[line width=2pt,draw=none];\node[fill=none,draw=none]{}; ] ] 
					[.{} [.{} \edge[line width=2pt,dotted];\node[fill=none,draw=none]{}; \edge[line width=2pt,draw=none];\node[fill=none,draw=none]{}; ] [.{} \edge[line width=2pt,dotted];\node[fill=none,draw=none]{}; \edge[line width=2pt,draw=none];\node[fill=none,draw=none]{}; ] ] 
				]
			    [.{} 
			    	[.{} [.{} \edge[line width=2pt,dotted];\node[fill=none,draw=none]{}; \edge[line width=2pt,draw=none];\node[fill=none,draw=none]{}; ] [.{} \edge[line width=2pt,dotted];\node[fill=none,draw=none]{}; \edge[line width=2pt,draw=none];\node[fill=none,draw=none]{}; ] ] 
			    	[.{} [.{} \edge[line width=2pt,dotted];\node[fill=none,draw=none]{}; \edge[line width=2pt,draw=none];\node[fill=none,draw=none]{}; ] [.{} \edge[line width=2pt,dotted];\node[fill=none,draw=none]{}; \edge[line width=2pt,draw=none];\node[fill=none,draw=none]{}; ] ] 
			    ]
			] 
		  ]
		  [.{} 
			[.{} [.{} [.{} [.{} \edge[line width=2pt,dotted];\node[fill=none,draw=none]{};\edge[line width=2pt,dotted];\node[fill=none,draw=none]{}; ] [.{} \edge[line width=2pt,dotted];\node[fill=none,draw=none]{};\edge[line width=2pt,dotted];\node[fill=none,draw=none]{}; ] ] [.{}  [.{} \edge[line width=2pt,dotted];\node[fill=none,draw=none]{};\edge[line width=2pt,dotted];\node[fill=none,draw=none]{}; ] [.{} \edge[line width=2pt,dotted];\node[fill=none,draw=none]{};\edge[line width=2pt,dotted];\node[fill=none,draw=none]{}; ] ] ]
			     [.{} [.{} [.{} \edge[line width=2pt,dotted];\node[fill=none,draw=none]{};\edge[line width=2pt,dotted];\node[fill=none,draw=none]{}; ] [.{} \edge[line width=2pt,dotted];\node[fill=none,draw=none]{};\edge[line width=2pt,dotted];\node[fill=none,draw=none]{}; ] ] [.{} [.{} \edge[line width=2pt,dotted];\node[fill=none,draw=none]{};\edge[line width=2pt,dotted];\node[fill=none,draw=none]{}; ] [.{} \edge[line width=2pt,dotted];\node[fill=none,draw=none]{};\edge[line width=2pt,dotted];\node[fill=none,draw=none]{}; ] ] ] 
			] 
			[.{} [.{} [.{} [.{} \edge[line width=2pt,dotted];\node[fill=none,draw=none]{};\edge[line width=2pt,dotted];\node[fill=none,draw=none]{}; ] [.{} \edge[line width=2pt,dotted];\node[fill=none,draw=none]{};\edge[line width=2pt,dotted];\node[fill=none,draw=none]{}; ] ] [.{} [.{} \edge[line width=2pt,dotted];\node[fill=none,draw=none]{};\edge[line width=2pt,dotted];\node[fill=none,draw=none]{}; ] [.{} \edge[line width=2pt,dotted];\node[fill=none,draw=none]{};\edge[line width=2pt,dotted];\node[fill=none,draw=none]{}; ] ] ]
			     [.{} [.{} [.{} \edge[line width=2pt,dotted];\node[fill=none,draw=none]{};\edge[line width=2pt,dotted];\node[fill=none,draw=none]{}; ] [.{} \edge[line width=2pt,dotted];\node[fill=none,draw=none]{};\edge[line width=2pt,dotted];\node[fill=none,draw=none]{}; ] ] [.{} [.{} \edge[line width=2pt,dotted];\node[fill=none,draw=none]{};\edge[line width=2pt,dotted];\node[fill=none,draw=none]{}; ] [.{} \edge[line width=2pt,dotted];\node[fill=none,draw=none]{};\edge[line width=2pt,dotted];\node[fill=none,draw=none]{}; ] ] ] 
			] 
		  ]
	     ]
        ]
      ]
\draw[dashed] (-12,-40pt)--++(43,0);
\draw (-12,-80pt)--++(43,0) node[right]{$n_1=2$};
\draw[dashed] (-12,-120pt)--++(43,0);
\draw (-12,-160pt)--++(43,0) node[right]{$n_2=4$};
\draw[dashed] (-12,-200pt)--++(43,0);
\draw[dashed] (-12,-240pt)--++(43,0);
\draw (-12,-280pt)--++(43,0) node[right]{$n_3=7$};
\end{tikzpicture}
\end{center}
\caption{Illustration of the proof of Lemma~\ref{lemma:tree2}} with $d(n)=2^{n-k}+\sum_{i=1}^k2^{n_i-i}$ for $n_k\leqslant n<n_{k+1}$.
\[
\begin{array}{ccl}
	d(0)&=&2^{0-0}=1\\
	d(1)&=&2^{1-0}=2\\
	d(2)&=&2^{2-0}=4\\
	d(3)&=&2^{3-1}+2^{2-1}=6\\
	d(4)&=&2^{4-1}+2^{2-1}=10\\
	d(5)&=&2^{5-2}+2^{2-1}+2^{4-2}=14\\
	d(6)&=&2^{6-2}+2^{2-1}+2^{4-2}=22\\
	d(7)&=&2^{7-2}+2^{2-1}+2^{4-2}=38\ldots
\end{array}
\]\label{fig:tree2}
\end{figure}
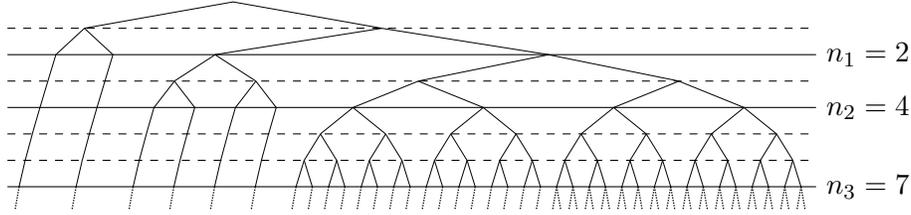

\begin{Lem}\label{lemma:tree2}
	For all $f:\NN\to\NN^*$ such that $f(n)=o(2^n)$, there exists a countable set $\mathcal D$ with growth function $d(n)=f(n),n\ge1$.
\end{Lem}

\begin{proof}
	Let $n_k,{k\ge0}$ be a strictly increasing sequence defined as follows: $n_0=0$ and for any $k\ge1$, $f(n)\leqslant 2^{n-k}$ for $n\geqslant n_k$. 
	Consider the set 
	\[
	\mathcal D:=\cup_{k\ge0}\{\underline 0w01^k:w\in A^{n_{k+1}-n_k}\}.
	\]
	This set is illustrated in Figure~\ref{fig:tree2}.
	Clearly, by definition, this set is countable. Moreover,  $d(n)\geqslant 2^{n-k}$ if $n_k\leqslant n<n_{k+1}$ (more precisely $d(n)=2^{n-k}+\sum_{i=1}^k2^{n_i-i}$) thus $d(n)\geqslant f(n)$. Once again, we can modify the set by withdrawing elements of $\mathcal D$ so that $d(n)= f(n)$, keeping countability. 
	
\end{proof}

\appendix
\noindent{\bf Ackowledgement} We thank Sacha Friedli, Daniel Y. Takahashi, Arnaud Le Ny, Roberto Fern\'andez and  Leandro Cioletti for many interesting discussion on $g$-measures, and Noam Berger for  suggesting the example of Subsection \ref{sec:non-exist-positive}. We would also like to mention that part of the proof of Theorem \ref{theo:exist} is inspired on a proof we found in an unpublished note by Sacha Friedli (that is no more available on the web). 

SG was supported by FAPESP  (BPE: 2017/07084-6) and CNPq (PQ 312315/2015-5 and  Universal  462064/2014-0). This work is part of the PhD thesis of RFF, supported by FAPESP Fellowship (2016/12918-0) and partially by FAPESP post-doctoral fellowship (2018/25076-5). This paper also was produced as part of the activities of FAPESP Research, Innovation and Dissemination Center for Neuromathematics (2013/07688-0).

\bibliographystyle{jtbnew} 
\bibliography{paccaut}

\vskip 10pt
\noindent Ricardo F. Ferreira \\
{\sc Departamento de Estat\'istica, \\Universidade Federal de S\~ao Carlos} \\
{\tt rferreira@ufscar.br}
\vskip 10pt
\noindent Sandro Gallo \\
{\sc Departamento de Estat\'istica, \\Universidade Federal de S\~ao Carlos} \\
{\tt sandro.gallo@ufscar.br}
\vskip 10pt
\noindent Fr\'ed\'eric Paccaut \\
{\sc Laboratoire Ami\'enois de Math\'ematiques Fondamentales et Appliqu\'ees cnrs umr 7352,} \\
{\sc Universit\'e de Picardie Jules Verne} \\
{\tt frederic.paccaut@u-picardie.fr}

\end{document}